\documentclass[12pt,reqno]{amsart}

\usepackage{color, MnSymbol, bm}
\usepackage{multirow}
\usepackage{changes}
\newcommand{\stkout}[1]{\ifmmode\text{\sout{\ensuremath{#1}}}\else\sout{#1}\fi}
\setdeletedmarkup{\stkout{#1}}

\usepackage{natbib}
\usepackage{hyperref}
\usepackage[latin1]{inputenc}
\usepackage{graphicx}
\usepackage{bbm}
\usepackage{subcaption}
\usepackage[autostyle]{csquotes}
\usepackage{siunitx}
\usepackage{float}
\usepackage{colortbl}

\numberwithin{equation}{section}

\theoremstyle{plain}                    % default
\newtheorem{thm}{Theorem}[section]
\newtheorem{lem}{Lemma}[section]

\newtheorem{prop}{Proposition}[section]

\theoremstyle{definition}

\theoremstyle{remark}
\newtheorem{rem}{Remark}               %\renewcommand{\therem}{}

\newtheorem{ack}{Acknowledgment\!\!}

\def\N{{\mathbb N}}
\def\Z{{\mathbb Z}}
\def\R{{\mathbb R}}
\def\P{{\mathbb P}}
\def\E{{\mathbb E}}
\def\F{{\mathcal F}}
\def\1{{\mathbbm{1}}}
\def\wtl{{\widetilde{\lambda}}}
\def\wty{{\widetilde{Y}}}
\def\wtx{{\widetilde{X}}}
\def\wtz{{\widetilde{Z}}}
\def\ninfty{\mathop{\longrightarrow}\limits_{n\to\infty}}
\def\tinfty{\mathop{\longrightarrow}\limits_{t\to\infty}}

\def\argmin{\mbox{arg}\,\min}
\newcommand{\var}{\mathop{\rm var}\nolimits}
\newcommand{\cov}{\mathop{\rm cov}\nolimits}
\newcommand{\Cov}{\mathop{\rm Cov}\nolimits}

\newcommand{\be}{\begin{equation}}
\newcommand{\bd}{\begin{displaymath}}
\newcommand{\ed}{\end{displaymath}}
\newcommand{\bea}{\begin{eqnarray}}
\newcommand{\eea}{\end{eqnarray}}
\newcommand{\bean}{\begin{eqnarray*}}
\newcommand{\eean}{\end{eqnarray*}}

\begin{document}

\thispagestyle{empty}

\baselineskip13pt

\begin{center}
{\large \sc Mixing properties of nonstationary multivariate count processes}
\end{center}

\vspace*{1cm}

\begin{center}
Zinsou Max Debaly\\
Universit\'e du Qu\'ebec en Abitibi-T\'emiscamingue, GREMA.\\
341 Rue Principale N,\\
Amos, QC J9T 2L8\\
Quebec, Canada\\
Email: debz01@uqat.ca\\[1cm]

Michael H.~Neumann${}^*$\\
Friedrich-Schiller-Universit\"at Jena\\
Institut f\"ur Mathematik\\
Ernst-Abbe-Platz 2\\
D -- 07743 Jena\\
Germany\\
Email: michael.neumann@uni-jena.de\\[1cm]
Lionel Truquet\\
ENSAI\\
Campus de Ker Lann\\
51 Rue Blaise Pascal\\
BP 37203\\
35172 Bruz Cedex\\
France\\
Email: lionel.truquet@ensai.fr\\[1.5cm]
\end{center}

\begin{center}
{\bf Abstract}
\end{center}
We consider multivariate versions of two popular classes of integer-valued processes.
While the transition mechanism is time-homogeneous, a possible nonstationarity is introduced
by an exogeneous covariate process.
We prove absolute regularity ($\beta$-mixing) for the count process with exponentially decaying
mixing coefficients. The proof of this result makes use of some sort of contraction in 
the transition mechanism which allows a coupling of two versions of the count process
such that they eventually coalesce.
We show how this result can be used to prove asymptotic normality of a least squares estimator 
of an involved model parameter.
\vspace*{1cm}

\footnoterule \noindent {\sl 2010 Mathematics Subject
Classification:} Primary 60G10; secondary 60J05. 
{\sl Short title:} Nonstationary multivariate count processes. \vfill
\noindent
${}^*$Correspondence to: Email: michael.neumann@uni-jena.de
%version: \today

\newpage

\setcounter{page}{1}
\pagestyle{headings}
\normalsize

%%%%%%%%%%%%%%%%%%%%%%%%%%%%%%%%%%%%%%%%%%%%%%%%%%%%%%%%%%%%%%%%%%%%%%%%%%%%%%%
\section{Introduction}
\label{S1}
%%%%%%%%%%%%%%%%%%%%%%%%%%%%%%%%%%%%%%%%%%%%%%%%%%%%%%%%%%%%%%%%%%%%%%%%%%%%%%%

Stationary time series models of counts have been extensively studied in the literature. In particular, many 
autoregressive models have been developed. See for instance \citet{Al1} and \citet{Al2} for INAR processes, \citet{Zeg88Markov},
\citet{Davis1}, \citet{Lat2006} or \citet{Fok2009} for Poisson autoregressive models, and the negative binomial INGARCH introduced in \citet{Zhu2011}.
Though most of the models discussed in the literature are univariate, a few contributions also considered multivariate time series of counts;
see in particular \cite{Lat97} for a multivariate extension of INAR processes, \citet{JLSS99} for state-space models or \citet{FSTD20} for a multivariate extension of Poisson autoregressive models. See also \citet{DT19} for some conditions ensuring existence of stationary solutions for 
some of these models. We defer the reader to \citet{Fok24} for a survey of some existing attempts to model multivariate time series of counts.
In a nonstationary framework, see also the recent contribution of \citet{WW18} for a multivariate time series model with common factors.

More recent contributions considered some stationary versions of the models discussed above but defined conditionally on some exogenous covariates.
See for instance \citet{dAg}, \citet{AF1}, \citet{DT21} or \citet{ALL22} in the context of univariate time series models of counts and \citet{DT23}
for generic multivariate time series models. See also \citet{DNT} for very general model specifications with strictly exogenous covariates.

Most of the references mentioned above focus on stationarity and ergodicity properties of the corresponding autoregressive models, since these key stability properties are often sufficient to derive consistency and asymptotic normality of maximum likelihood estimators. However, such a stationary modeling is only relevant when the covariates exhibit itself a stationary behavior. When these covariates are not time-stationary, with a possible explosive behavior, such as a polynomial trend, the model is by essence non-stationary with possible explosive moments. Statistical inference of autoregressive parameters is then more complicated in this setting and mixing properties rather than ergodicity are in order to derive 
consistency and asymptotic normality of estimators.
Recently, \citet{DLN22} derived mixing properties of a Poisson INGARCH model with non-stationary covariates and applied their results to statistical testing for existence of a trend in the intensity of the process. \citet{LN23} also proved such properties for log-linear count processes
where the conditional distributions are taken from a nearly scale-invariant family which allows for even stronger trends. 
We complement their results about mixing properties by considering INGARCH and INAR models, both in the univariate and in the multivariate case. Additionally, we apply our mixing properties to study the asymptotic properties of the least squares estimator in an INAR$(1)$ model in which an additive explosive sequence of Poisson covariates is incorporated in the dynamic. Non-standard fast rates of convergence are obtained,
depending on the growth of the intensity of the Poisson covariate. Our contribution justifies the importance of studying non-stationary models in the context of count data. Indeed, contrarily to the case of continuous data modeled by linear models, it is not possible to detrend the data by differentiation or to decompose the response as a sum of a non-stationary and a stationary component. In this sense, working with a non-stationary count process requires a specific attention and a careful analysis of the asymptotic properties of some classical estimators. Our contribution is one step in this direction.

The paper is organized as follows. In Section \ref{S2}, we present our main results for absolute regularity of two multivariate time series models of counts, a multivariate version of INGARCH model and a multivariate INAR model with nonstationary covariates. We apply our results to statistical inference in a non-stationary INAR model in Section \ref{S3}, followed by a numerical analysis and an application to real data.
Proofs of our results are postponed to the last two sections of the paper.

%%%%%%%%%%%%%%%%%%%%%%%%%%%%%%%%%%%%%%%%%%%%%%%%%%%%%%%%%%%%%%%%%%%%%%%%%%%%%%%
\section{Main results}
\label{S2}
%%%%%%%%%%%%%%%%%%%%%%%%%%%%%%%%%%%%%%%%%%%%%%%%%%%%%%%%%%%%%%%%%%%%%%%%%%%%%%%

\subsection{Two models for multivariate count processes}

We consider the following two classes of processes:
\begin{itemize}
\item[1)\quad] {\bf Multivariate Poisson-INGARCH(1,1)}
\end{itemize}
We assume that $(X_t)_{t\in\N_0}$ is a $d$-variate process on a probability space $(\Omega,\F,P)$ obeying the model equations
\begin{subequations}
\begin{eqnarray}
X_t \mid \F_{t-1} & \sim & \mbox{Poi}(\lambda_t), \label{m1.a}\\
\lambda_t & = & A\lambda_{t-1} \,+\, B X_{t-1} \,+\, Z_{t-1} \label{m1.b}
\end{eqnarray}
\end{subequations}
for all $t\in\N$, where $\F_s=\sigma(\lambda_0,X_0,Z_0,\ldots,X_s,Z_s)$ and $(Z_t)_{t\in\N_0}$
is a sequence of independent covariates, $Z_t$ being independent of $\F_{t-1}$ and $X_t$.
For $\lambda_t=(\lambda_{t,1},\ldots,\lambda_{t,d})^T$, $\mbox{Poi}(\lambda_t)$
is a $d$-dimensional Poisson distribution with independent components
and respective intensities $\lambda_{t,1},\ldots,\lambda_{t,d}$.
$A$ and $B$ are $(d\times d)$-matrices with non-negative entries.
\begin{itemize}
\item[2)\quad] {\bf Multivariate GINAR(1)}
\end{itemize}
Multivariate generalized INAR (GINAR) processes were introduced by \citet{Lat97}.
In this case we assume that $(X_t)_{t\in\N_0}$ is a $d$-variate process,
where its components follow the equations
\begin{subequations}
\begin{equation}
\label{m2.a}
X_{t,i} \,=\, \sum_{j=1}^d \sum_{s=1}^{X_{t-1,j}} Y_{t,s}^{i,j} \,+\, Z_{t,i}, \qquad i=1,\ldots,d, \;\; t\in\N,
\end{equation}
where $(Z_t)_{t\in\N}$ is a sequence of independent random vectors, $Z_t=(Z_{t,1},\ldots,Z_{t,d})^T$.
This sequence is assumed to be independent of the collection of count variables
$\big\{ Y_{t,s}^{i,j}\colon\; (t,s,i,j)\in\N^2\times\{1,\ldots,d\}^2 \big\}$
which itself is composed of independent random variables such that
\begin{equation}
\label{m2.b}
Y_{t,s}^{i,j} \,\sim\, \mbox{Bin}(1, B_{ij}),
\end{equation}
where $\mbox{Bin}(n,p)$ denotes a binomial distribution with parameters~$n$ and~$p$; in the special case of $n=1$
this is also known as Bernoulli distribution.
With $B:=(\!(B_{ij})\!)_{i,j=1,\ldots,d}$, it is customary to use the compact notation
\begin{equation}
\label{m2.c}
X_t \,=\, B\circ X_{t-1} \,+\, Z_t, \qquad t\in\N,
\end{equation}
\end{subequations}
where $\circ$ is a $(d\times d)$-dimensional version of the thinning operator.

In the Poisson-INGARCH model, the vector $\lambda_t=(\lambda_{t,1},\ldots,\lambda_{t,d})^T$
describes the current state of the process at time~$t$.
For the GINAR model, this role is taken by the vector of count variables
$X_{t-1}=(X_{t-1,1},\ldots,X_{t-1,d})^T$.
Since we are mainly interested in nonstationary processes, there is no gain
by considering two-sided versions $(X_t)_{t\in\Z}$. Rather, the processes are started
at time~0 with an initial intensity parameter $\lambda_0=(\lambda_{0,1},\ldots,\lambda_{0,d})^T$
or with initial values of the count variables $X_{0,1},\ldots,X_{0,d}$, respectively.
Note that there is a certain connection between these two models.
If in (\ref{m2.b}) $Y_{t,s}^{i,j}\sim\mbox{Poi}(B_{ij})$ and if in (\ref{m2.c}) $Z_t\sim\mbox{Poi}(c_t)$ for a constant~$c_t$,
then
\begin{displaymath}
X_t\mid \F_{t-1} \,\sim\, \mbox{Poi}\big( BX_{t-1} \,+\, c_t \big),
\end{displaymath}
which corresponds to a Poisson-INGARCH model with $A=0_{d\times d}$ (Poisson-INARCH).

We intend to prove absolute regularity for the count process $(X_t)_{t\in\N_0}$, where we allow in particular 
an arbitrarily strong trend. We make use of some sort of contraction in the
transition mechanism which will follow if the entries of the respective matrices~$A$ 
and~$B$ are sufficiently small.
In contrast, the exogenous variables $Z_{t,i}$ may cause a trend. Note that a commonly assumed intercept may also be hidden in these covariates.
Natural choices
for their distributions are 
\begin{displaymath}
Z_{t,i} \,\sim\, \mbox{Bin}(N_{t,i},p) \qquad \mbox{ or } \qquad
Z_{t,i} \,\sim\, \mbox{Poi}(\gamma_{t,i}) \qquad \mbox{ or simply } \qquad Z_{t,i}=N_{t,i}, 
\end{displaymath}
where the constants $N_{t,i}$ and $\gamma_{t,i}$ may increase without bound as $t$ tends to infinity.
In the special case of a univariate Poisson-INGARCH(1,1) process with trend,
mixing properties were derived in \citet[Section~3.1]{DLN22}.

\subsection{A general method for proving absolute regularity}

Let $(\Omega,{\mathcal A},P)$ be a probability space and ${\mathcal A}_1$, ${\mathcal A}_2$
be two sub-$\sigma$-algebras of ${\mathcal A}$. Recall that the coefficient of absolute regularity is defined as
\bd
\beta({\mathcal A}_1,{\mathcal A}_2) \,=\, E\big[ \sup\{ |P(B\mid {\mathcal A}_1) \,-\, P(B)|\colon
\;\; B\in {\mathcal A}_2 \} \big].
\ed
For a process $(X_t)_{t\in\N_0}$ on $(\Omega,\F,P)$, the
coefficients of absolute regularity at the point~$k$ are defined as
\bd
\beta^X(k,n) \,=\, \beta\big( \sigma(X_0,X_1,\ldots,X_k), \sigma(X_{k+n},X_{k+n+1},\ldots) \big)
\ed
and the (global) coefficients of absolute regularity as
\bd
\beta^X(n) \,=\, \sup\{ \beta^X(k,n)\colon \;\; k\in\N_0\}.
\ed

The intended approach of proving absolute regularity is inspired by the fact that one can construct,
on a suitable probability space $(\widetilde{\Omega},\widetilde{\F},\widetilde{P})$ two versions
$(\wtx_t)_{t\in\N_0}$ and $(\wtx_t')_{t\in\N_0}$ of the process $(X_t)_{t\in\N_0}$ such that
$(\wtx_0,\ldots,\wtx_k)$ and $(\wtx_0',\ldots,\wtx_k')$ are independent and
\be
\label{1.1}
\beta^X(k,n) \,=\, \widetilde{P}\left( \wtx_{k+n+r}\neq\wtx_{k+n+r}' \mbox{ for some } r\geq 0 \right).
\end{equation}
A representation of $\beta^X(k,n)$ as in (\ref{1.1}) requires that we find an optimal coupling for the
complete sequences $(\widetilde{X}_{k+1},\widetilde{X}_{k+2},\ldots)$ and $(\widetilde{X}_{k+1}',\widetilde{X}_{k+2}',\ldots)$.
Finding an explicit description of such a coupling seems to be out of reach in our context.
Rather, we employ a maximal coupling of $\widetilde{X}_t$ and $\widetilde{X}_t'$ at each single time point $t\geq k+n$
which still leads to good upper estimates of the coefficients~$\beta^X(k,n)$.
Actually, if $(\wtx_t)_{t\in\N_0}$ and $(\wtx_t')_{t\in\N_0}$ defined on a common probability space
$(\widetilde{\Omega},\widetilde{\F},\widetilde{P})$ are any two versions of $(X_t)_{t\in\N_0}$
such that $(\wtx_0,\ldots,\wtx_k)$ and $(\wtx_0',\ldots,\wtx_k')$ are independent and, of course,
$\widetilde{P}^{\widetilde{X}_{k+r}| \widetilde{X}_{k+r-1}=x_1, \widetilde{X}_{k+r-2}=x_2,\ldots} 
=P^{X_{k+r}| X_{k+r-1}=x_1, X_{k+r-2}=x_2,\ldots}$ and
$\widetilde{P}^{\widetilde{X}_{k+r}'| \widetilde{X}_{k+r-1}'=x_1, \widetilde{X}_{k+r-2}'=x_2,\ldots} 
=P^{X_{k+r}| X_{k+r-1}=x_1, X_{k+r-2}=x_2,\ldots}$ for all $r\geq 1$, then
\bea
\label{1.2}
\lefteqn{ \beta^X(k,n) } \nonumber \\
& \leq & \widetilde{E}\Big[ \sup_{C\in\sigma({\mathcal C})} \big\{ \big|
\widetilde{P}\left( (\wtx_{k+n},\wtx_{k+n+1},\ldots)\in C\mid \wtx_0,\ldots,\wtx_k \right) \nonumber \\
& & \qquad \qquad \qquad
\,-\, \widetilde{P}\left( (\wtx_{k+n}',\wtx_{k+n+1}',\ldots)\in C\mid \wtx_0',\ldots,\wtx_k' \right) \big| \big\} \Big]
\qquad \nonumber \\
& \leq & \widetilde{P}\left( \wtx_{k+n+r} \neq \wtx_{k+n+r}' \quad \mbox{for some } r\in\N_0 \right) \nonumber \\
& = & \widetilde{P}\left( \wtx_{k+n} \neq \wtx_{k+n}' \right) \nonumber \\
& & {} \,+\, \sum_{r=1}^\infty \widetilde{P}\left( \wtx_{k+n+r} \neq \wtx_{k+n+r}', \wtx_{k+n+r-1} = \wtx_{k+n+r-1}',
\ldots, \wtx_{k+n} = \wtx_{k+n}' \right). 
\eea
(In the second line of this display, $\sigma({\mathcal C})$ denotes the $\sigma$-algebra generated by the cylinder sets.)
In the special case when $(X_t)_{t\in\N_0}$ is a Markov chain, we obtain the following simpler estimate for the
mixing coefficient:
\begin{equation}
\label{1.2a}
\beta^X(k,n) \,=\, \beta\big( \sigma(X_k), \sigma(X_{k+n}) \big)
\,\leq\, \widetilde{P}\left( \wtx_{k+n} \neq \wtx_{k+n}' \right).
\end{equation}

\subsection{Absolute regularity of multivariate count processes}

First we set some notation. 
For a vector~$x=(x_1,\ldots,x_d)^T$ and $p=1,2$, we denote by $\|x\|_p=\big(\sum_{j=1}^d |x_i|^p\big)^{1/p}$ $L_p$ norm.
Correspondingly, for a $(d\times d)$-matrix~$M$, $\|M\|_1=\max\big\{\|Mx\|_1\colon\, \|x\|_1=1\big\}=\max_j \big\{ \sum_{i=1}^d |M_{ij}| \big\}$
and 
$\rho(M)=\max\big\{|\lambda|\colon\, \lambda \mbox{ eigenvalue of }M\big\}$ denote its maximum absolute column sum norm and its spectral radius, respectively.
Note that~$\rho(M)$ should not be mixed up with the spectral norm
$\|M\|_2=\max\big\{\|Mx\|_2\colon\, \|x\|_2=1\big\}=\max\big\{|\lambda|\colon\, \lambda \mbox{ singular value of }M\big\}$.
$\rho(M)$ is the infimum of all possible matrix norms~$\|M\|$ and therefore less than or equal to $\|M\|_2$.
Furthermore, for a matrix~$M$ and a vector~$x$ with non-negative components, $\sqrt{M}$ and $\sqrt{x}$ denote
the corresponding objects with respective components~$\sqrt{M_{ij}}$ and~$\sqrt{x_i}$.
$d_{TV}(P,Q)=(1/2)\sum_{k\in\N_0}|P(\{k\})-Q(\{k\})|$ denotes the total variation norm between
two generic distributions~$P$ and~$Q$ on $(\N_0,2^{N_0})$.

\subsubsection{The INGARCH(1,1) process}

To give some insight into details of our approach, we consider first the multivariate Poisson-INGARCH model (\ref{m1.a}),
(\ref{m1.b}).
In this case, the construction leading to a good estimate (\ref{1.2}) may be divided into three phases and
we sketch in what follows how this is accomplished.

{\bf Phase~1:}\\
First we show for independent versions $((\wtx_t,\wtz_t))_{t\in\N_0}$ and $((\wtx_t',\wtz_t'))_{t\in\N_0}$ of the process
$((X_t,Z_t))_{t\in\N_0}$ that
\begin{equation}
\label{ph1}
\sup_k \Big\{ \widetilde{E} \big\| \sqrt{\wtl_k} \,-\, \sqrt{\wtl_k'} \big\|_1 \Big\} \,<\, \infty.
\end{equation}

{\bf Phase~2:}\\
Note that the first term on the right-hand side of (\ref{1.2}) can be made as small as\\
$\widetilde{E} d_{TV} \big( P^{X_{k+n}\mid \lambda_{k+n}=\wtl_{k+n}}, P^{X_{k+n}\mid \lambda_{k+n}=\wtl_{k+n}'} \big)
\,=\, \widetilde{E} d_{TV}\big( \mbox{Poi}(\wtl_{k+n}), \mbox{Poi}(\wtl_{k+n}')\big)$.
Furthermore, since the components of $X_{k+n}$ are conditionally independent we may use the estimate
\begin{displaymath}
d_{TV}\big(\mbox{Poi}(\wtl_{k+n}), \mbox{Poi}(\wtl_{k+n}')\big)
\,\leq\, \sum_{i=1}^d d_{TV}\big(\mbox{Poi}(\wtl_{k+n,i}), \mbox{Poi}(\wtl_{k+n,i}')\big).
\end{displaymath}
A good guideline for our construction is given by the upper estimate
\begin{equation}
\label{1.3}
d_{TV}\big( \mbox{Poi}(\wtl_{k+n,i}), \mbox{Poi}(\wtl_{k+n,i}') \big)
\,\leq\, \sqrt{2/e}\;  \big| \sqrt{\wtl_{k+n,i}} \,-\, \sqrt{\wtl_{k+n,i}'} \big|,
\end{equation}
where~$e$ is Euler's number; see \citet[formula (5)]{Roo03} or Exercise~9.3.5(b) in \citet[page~300]{DvJ88}.
It turns out that this estimate is also suitable in case of a strong trend
and it is therefore used in our proofs. 
In view of (\ref{1.3}), we shall couple the variables $\wtx_{k+1},\wtz_{k+1},\ldots,\wtx_{k+n-1},\wtz_{k+n-1}$
with their corresponding counterparts $\wtx_{k+1}',\wtz_{k+1}',\ldots,\wtx_{k+n-1}',\wtz_{k+n-1}'$
such that $\widetilde{E}\big\| \sqrt{\wtl_{k+n}} \,-\, \sqrt{\wtl_{k+n}'} \big\|_1$ gets small as~$n$ grows.
Since the covariate~$Z_t$ is independent of ${\mathcal F}_{t-1}$ we choose for $t=k+1,\ldots,k+n-1$
$\wtz_t$ and $\wtz_t'$ such that they are equal. 
For the count variables we apply a (step-wise) maximal coupling; see Lemma~\ref{couplinglemma} below.
This implies in particular that the difference between $\wtx_t$ and $\wtx_t'$ also gets small,
\begin{displaymath}
\widetilde{E} \big\| \sqrt{\wtl_{k+n}} \,-\, \sqrt{\wtl_{k+n}'} \big\|_1
\,\leq\, \|(\sqrt{A}+2\sqrt{B})^{n-1}\|_1 \; \widetilde{E} \big\| \sqrt{\wtl_{k+1}} \,-\, \sqrt{\wtl_{k+1}'} \big\|_1
\,=\, O\big( \|(\sqrt{A}+2\sqrt{B})^{n-1}\|_1 \big).
\end{displaymath}

{\bf Phase~3:}\\
In the third phase, the focus is clearly on getting
$\widetilde{P}\big( \wtx_{k+n+r}\neq\wtx_{k+n+1} \quad \mbox{ for some } r\geq 0\big)$
small. 
For given $\wtl_{k+n}$ and $\wtl_{k+n}'$, we apply a maximal coupling to the corresponding components of
$\wtx_{k+n}$ and $\wtx_{k+n}'$, and we obtain
\begin{displaymath}
\widetilde{P}\big( \wtx_{k+n}\neq\wtx_{k+n}' \mid \wtl_{k+n}, \wtl_{k+n}' \big)
\,\leq\, \sum_{i=1}^d \widetilde{P}\big( \wtx_{k+n,i}\neq\wtx_{k+n,i}' \mid \wtl_{k+n}, \wtl_{k+n}' \big),
\end{displaymath}
and so
\begin{displaymath}
\widetilde{P}\big( \wtx_{k+n}\neq\wtx_{k+n}' \big)
\,\leq\, \sum_{i=1}^d \sqrt{2/e} \; \widetilde{E} \big| \sqrt{\wtl_{k+n,i}} - \sqrt{\wtl_{k+n,i}'} \big| \\
\,=\, O\Big( \big\| (\sqrt{A}+2\sqrt{B})^{n-1} \big\|_1 \Big).
\end{displaymath}
For $r>0$, we focus on  
$\widetilde{P}\big( \wtx_{k+n+r,i}\neq\wtx_{k+n+r,i}',\wtx_{k+n}=\wtx_{k+n}',\ldots,\wtx_{k+n+r-1}=\wtx_{k+n+r-1}' \big)$,
i.e.~we only have to consider the case of
$\wtx_{k+n}=\wtx_{k+n}',\ldots,\wtx_{k+n+r-1}=\wtx_{k+n+r-1}'$. We obtain that
\begin{eqnarray*}
\big\| \sqrt{\wtl_{k+n+r}} \,-\, \sqrt{\wtl_{k+n+r}'} \big\|_1
\;\1\big( \wtx_{k+n}=\wtx_{k+n}',\ldots,\wtx_{k+n+r-1}=\wtx_{k+n+r-1}' \big) \\
\leq\, \big\| (\sqrt{A})^r \big\|_1 \; \big\| \sqrt{\wtl_{k+n}} \,-\, \sqrt{\wtl_{k+n}'} \big\|_1.
\end{eqnarray*}
We apply again a maximal coupling to $\wtx_{k+n+r,i}$ and $\wtx_{k+n+r,i}'$, which leads to
\begin{eqnarray*}
\widetilde{P}\big( \wtx_{k+n+r}\neq\wtx_{k+n+r}',
\wtx_{k+n}=\wtx_{k+n}',\ldots,\wtx_{k+n+r-1}=\wtx_{k+n+r-1}' \big) \\
=\, O\Big( \big\| (\sqrt{A}+2\sqrt{B})^{n-1} \big\|_1 \; \big\| (\sqrt{A})^r \big\|_1 \Big).
\end{eqnarray*}
Recall that Gelfand's formula states that $\rho(M)\,=\,\lim_{n\to\infty} \|M^n\|^{1/n}$ for every square
matrix~$M$ and every matrix norm~$\|\cdot\|$.
This implies that
\begin{equation}
\label{l1-sr} 
\| M^n \| \,=\, O\big( \kappa^n \big) \qquad \forall \kappa>\rho(M).
\end{equation}
To summarize, we obtain that
\begin{eqnarray*}
\lefteqn{ \widetilde{P}\big( \wtx_{k+n}\neq\wtx_{k+n}' \big)
\,+\, \sum_{r=1}^\infty \widetilde{P}\big( \wtx_{k+n+r}\neq\wtx_{k+n+r}',
\wtx_{k+n}=\wtx_{k+n}',\ldots,\wtx_{k+n+r-1}=\wtx_{k+n+r-1}' \big) } \\
& = & O\Big( \big\| (\sqrt{A}+2\sqrt{B})^{n-1} \big\|_1 \; \sum_{r=0}^\infty \|(\sqrt{A})^r\|_1 \Big) 
\,=\, O\Big( \big(\rho(\sqrt{A}+2\sqrt{B})+\epsilon\big)^n \Big),
\end{eqnarray*}
for arbitrary $\epsilon>0$.

In view of this discussion, we impose the following condition.
\begin{itemize}
\item[{\bf (A1)}\quad]
\begin{itemize}
\item[(i)\quad] $\rho\left(\sqrt{A}+2\sqrt{B}\right)<1$,
\item[(ii)\quad] $E\sqrt{\lambda_{0,i}}\,<\,\infty$ \quad ($i=1,\ldots,d$),
\item[(iii)\quad] $\sup_{t\in\N} E\big| \sqrt{Z_{t,i}} \,-\, E \sqrt{Z_{t,i}} \big|\,<\,\infty$ \quad ($i=1,\ldots,d$).
\end{itemize}
\end{itemize}
\bigskip

{\rem
Note that (A1)(iii) is fulfilled if the $Z_{t,i}$s have a Poisson or binomial distribution.
Since $x\mapsto \sqrt{x}$ is a concave function we obtain for a non-negative random variable~$Z$ that
$E\sqrt{Z}\leq \sqrt{EZ}$. For $Z\sim\mbox{Poi}(\lambda)$, we obtain that
\begin{eqnarray*}
E\big| \sqrt{Z} \,-\, E\sqrt{Z} \big|
& = & E\big[ \big( E\sqrt{Z} - \sqrt{Z} \big)^+ \big] \,+\, E\big[ \big( E\sqrt{Z} - \sqrt{Z} \big)^- \big] \\
& = & 2\, E\big[ \big( E\sqrt{Z} - \sqrt{Z} \big)^+ \big] \,\leq\, 2\, E\big[ \big( \sqrt{EZ} - \sqrt{Z} \big)^+ \big] \\
& \leq & \frac{ 2\, E\big[ |Z \,-\, EZ| \, \1(\sqrt{Z}\leq \sqrt{EZ}) \big] }{ \sqrt{EZ} } \,\leq\, \frac{ 2\sqrt{\var(Z)} }{ \sqrt{EZ} } \,=\, 2.
\end{eqnarray*}
Likewise, if $Z\sim\mbox{Bin}(n,p)$, then
\begin{displaymath}
E\big| \sqrt{Z} \,-\, E\sqrt{Z} \big|
\,\leq\, \frac{ 2\sqrt{\var(Z)} }{ \sqrt{EZ} } \,=\, 2\,\sqrt{1-p}.
\end{displaymath}
}
\bigskip

Now we are in a position to state our first major result.
\bigskip

{\thm
\label{T1}
Suppose that (\ref{m1.a}), (\ref{m1.b}), and (A1) are fulfilled.
Then the count process $(X_t)_{t\in\N_0}$ is absolutely regular and the coefficients of absolute regularity satisfy
\begin{displaymath}
\beta^X(n) \,=\, O\big( \kappa^n \big),
\end{displaymath}
for any $\kappa>\rho\big(\sqrt{A}+2\sqrt{B}\big)$.
}
\bigskip
{\rem 
The calculations in the proofs of our theorems, and in particular the verification of a contraction property, are based
on 1-norms of vectors and matrices.
Note that, for a matrix~$C$, $\rho(C)$ is the infimum of $\|C\|$ over all possible matrix norms $\|\cdot\|$.
However, since 
$\big\|(\sqrt{A}+2\sqrt{B})^{n-1}\big\|_1=O\big((\rho(\sqrt{A}+2\sqrt{B})+\epsilon)^n\big)$ $\forall \epsilon>0$
it follows that our condition {\bf (A1)}(i) on the spectral radius suffices.
                                                           
As one of the reviewers remarked, this condition may still appear to be quite restrictive.
On the other hand, when $B$ is the zero matrix, then we obtain the mixing property for such models
with explosive covariates as soon as the spectral radius $\rho\big(\sqrt{A}\big)$ is less than~1.
When the dimension~$d$ is~1, this corresponds to $A_{11}<1$ which is a necessary condition even in the stationary case.

Another interesting case is when $d=2$ and the first count time series is autonomous, i.e.~$A_{12}=B_{12}=0$.
Then $\sqrt{A}+2\sqrt{B}$ is a triangular matrix and it follows from
$\mbox{det}\big(\lambda\, I_2-(\sqrt{A}+2\sqrt{B})\big)
=\big(\lambda-(\sqrt{A_{11}}+2\sqrt{B_{11}})\big)\big(\lambda-(\sqrt{A_{22}}+2\sqrt{B_{22}})\big)$
that their eigenvalues are the entries on the main diagonal.
Therefore
\begin{displaymath}
\max_{i=1,2} \Big\{ \sqrt{A_{ii}} \,+\, 2\,\sqrt{B_{ii}} \big\} \,<\, 1
\end{displaymath}
is sufficient and the values of $A_{21}$ and $B_{21}$ can be arbitrarily large.
This example drastically shows why it is advantageous to impose a condition on the spectral radius
rather than on the 1-norm of $\sqrt{A}+2\sqrt{B}$.
}
\bigskip
 
{\rem
(i)\quad
The result of Theorem~\ref{T1} can be easily generalized for some class of Poisson-INGARCH processes
where the components conditioned on the past are not necessarily independent.
Suppose that $(X_t)_{t\in\N_0}$ is a $d$-variate process with components conditioned on the past being independent, and that
(\ref{m1.a}), (\ref{m1.b}), and (A1) are fulfilled.
Suppose further that~$H$ is an invertible $(d\times d)$-matrix with entries~$H_{ij}$ being either~0 or~1, and let
\begin{displaymath}
Y_t \,=\, \big(Y_{t,1},\ldots,Y_{t,d}\big)^T \,=\, H\, X_t \qquad \forall t\in\N_0.
\end{displaymath}
Then
\begin{displaymath}
Y_{t,i}\mid\F_{t-1} \,\sim\, \mbox{Poi}\big( \lambda_{t,i}^Y \big),
\end{displaymath}
where $\lambda_t^Y=(\lambda_{t,1}^Y,\ldots,\lambda_{t,d}^Y)^T=H\lambda_t$,
and 
\begin{displaymath}
\Cov\big( Y_t \big| \F_{t-1} \big) \,=\, H\, \mbox{Diag}\big(\lambda_{t,1},\ldots,\lambda_{t,d}\big) \, H^T.
\end{displaymath}
It follows that the intensity process $\big(\lambda_t^Y\big)_{t\in\N}$ also satisfies an
equation similar to~(\ref{m1.b}),
\begin{displaymath}
\lambda_t^Y \,=\, H\,A\,H^{-1}\,\lambda_{t-1}^Y \,+\, H\,B\,H^{-1}\,Y_{t-1} \,+\, H\,Z_{t-1} \qquad \forall t\in\N.
\end{displaymath}
Since $Y_t$ is a function of $X_t$ we have that $\sigma(Y_0,\ldots,Y_k)\subseteq\sigma(X_0,\ldots,X_k)$
and\\
$\sigma(Y_{k+n},Y_{k+n+1},\ldots)\subseteq\sigma(X_{k+n},X_{k+n+1},\ldots)$. This  implies that
\begin{displaymath}
\beta^Y\big(k,n) \,\leq\, \beta^X\big(k,n\big) \qquad \forall k,n\in\N_0,
\end{displaymath}
and so
\begin{displaymath}
\beta^Y\big(n) \,\leq\, \beta^X\big(n\big) \qquad \forall n\in\N_0,
\end{displaymath}
that is, the process $(Y_t)_{t\in\N_0}$ inherits the property of absolute regularity from
the underlying process $(X_t)_{t\in\N_0}$.
A special case is given by a matrix~$H$ with entries 
$H_{ij}=1$ if $i=j$ or $j=d$, and $H_{ij}=0$ if $i\neq j$ and $j<1$.
It follows from $\mbox{det}(H)=1$ that this matrix is invertible.
We obtain for the first $d-1$ components of the vector~$Y_t$ that
\begin{displaymath}
Y_{t,i} \,=\, X_{t,i} \,+\, X_{t,d}.
\end{displaymath}
This is a popular and simple method to define a multivariate distribution with the property that
the marginal distribution of each variable is Poisson which goes back to \citet{Cam34} and \citet{Tei54}.
One limitation of this approach is that it permits only non-negative dependencies between the components.
This follows from the fact that all components are sums of independent Poisson random variables,
and hence, covariances are governed by the non-negative intensity parameters~$\lambda_{t,i}$.

(ii)\quad
\citet{FSTD20} proposed the following approach to generate multivariate Poisson distributions which also allows 
for negative correlations. For given $\lambda_t=(\lambda_{t,1},\ldots,\lambda_{t,d})'$,
samples $U_{t,l}=(U_{t,l,1},\ldots,U_{t,l,d})'$ ($l=1,2,\ldots$) from a $d$-dimensional copula~$C$ were
transformed to marginally exponentially distributed variates $X_{t,l,i}=-\ln(U_{t,l,i})/\lambda_{t,i}$.
Then marginally Poisson distributed variates with intensity~$\lambda_{t,i}$ are obtained by 
$Y_{t,i}=\max\{k\colon \sum_{l=1}^k X_{t,l,i}\leq 1\}$.
However, we were not able to prove absolute regularity of explosive variants of such processes
since we could not find an efficient estimate of the total variation between two such
distributions in terms of the square roots of the intensities. Moreover, our sufficient condition for
the mixing properties of this model
remains stronger than the optimal stationarity condition $\rho(A+B)<1$ given in \citet{DT19}.
These restrictions are mainly due to our proof and assumptions allowing explosive covariates in the model.
When the covariates are nonstationary but with a non-explosive mean, one can derive mixing properties
of the model without conditional independence or parameter restrictions. We give a result in what follows.
}
\bigskip

In the following result, we still consider model (\ref{m1.b}) but now, for given $\lambda_t=(\lambda_{t,1},\ldots,\lambda_{t,d})'$,
we construct dependent Poisson variates according to the general construction of a multivariate count distribution with Poisson marginals
and an underlying copula~$C$ as presented in \citet{FSTD20}.
We start with samples $U_{t,l}=(U_{t,l,1},\ldots,U_{t,l,d})'$ ($l=1,2,\ldots$) from a $d$-dimensional copula~$C$
and transform these random variables to $X_{t,l,i}=-\ln(U_{t,l,i})\sim\mbox{Exp}(1)$.
In contrast to \citet{FSTD20} who directly generated Poisson distributed variates, we first generate Poisson processes
$N_u^{(t,i)}=\max\{k\in\N\colon\, \sum_{l=1}^k X_{t,i,l}\leq u\}$.
Then $(N_u^{(t,i)})_{u\geq 0}$ is a Poisson process with intensity~1 and, for $\lambda_t=(\lambda_{t,1},\ldots,\lambda_{t,d})'$,
$(N_{\lambda_{t,1}}^{(t,1)},\ldots,N_{\lambda_{t,d}}^{(t,d)})'$
has conditionally a Poisson distribution with dependent components. We denote the distribution of this vector by $\mbox{Poi}_{dep}(\lambda_t)$.
\bigskip

{\thm
\label{T2+}
Suppose that $\rho(A+B)<1$, $\E\lambda_{0,i}<\infty$ and $\sup_{t\in\N}\E\vert Z_{t,i}\vert<\infty$ for $i=1,\ldots,d$.
Then the count process $(X_t)_{t\in\N_0}$ is absolutely regular and the coefficients of absolute regularity satisfy
\begin{displaymath}
\beta^X(n) \,=\, O\big( \kappa^n \big),
\end{displaymath}
for any $\kappa>\rho(A+B)$.
}
\bigskip

{\rem One can show directly that the condition $\rho(A+B)<1$ is weaker than {\bf (A1)}(i).
Let~$C$ and~$D$ be two square matrices of size $d\times d$ and with non-negative coefficients.
Since $\|C^n\|_1\leq \|(\sqrt{C})^n\|_1^2$ we obtain from Gelfand's formula
\begin{displaymath}
\sqrt{\rho(C)} \,=\, \lim_{n\rightarrow \infty}\Vert C^n\Vert_1^{1/2n}
\,\leq\, \lim_{n\rightarrow \infty}\Vert (\sqrt{C})^n\Vert_1^{1/n} \,=\, \rho\left(\sqrt{C}\right).
\end{displaymath}
Moreover, if $C_{ij}\leq D_{ij}$ for $1\leq i,j\leq d$, then $\rho(C)\leq\rho(D)$.
For the matrices $A$ and $B$ defining our model, we then deduce that
\begin{displaymath} 
\sqrt{\rho(A+B)} \,\leq\, \rho\left(\sqrt{A+B}\right) \,\leq\, \rho\left(\sqrt{A}+2\sqrt{B}\right).
\end{displaymath}
When the covariates are integrable and non-explosive, Theorem~\ref{T2+} provides absolute regularity of the count process
under much less restrictive assumptions. 
}
\bigskip

\subsubsection{The GINAR(1) process}

Now we turn to the GINAR model (\ref{m2.a}), (\ref{m2.b}). In this case, the count process $(X_t)_{t\in\N_0}$
is Markovian, which means that the coefficients of absolute regularity can be estimated according to
(\ref{1.2a}).
Here we impose the following condition.
\begin{itemize}
\item[{\bf (A2)}\quad]
\begin{itemize}
\item[(i)\quad] $\rho\left(\sqrt{B}\right)\,<\,1/2$,
\item[(ii)\quad] $E\sqrt{X_{0,i}}\,<\,\infty$ \quad ($i=1,\ldots,d$),
\item[(iii)\quad] $\sup_{t\in\N} E\big| \sqrt{Z_{t,i}} \,-\, E \sqrt{Z_{t,i}} \big|\,<\,\infty$ \quad ($i=1,\ldots,d$).
\end{itemize}
\end{itemize}
Now we can state our second main result.

{\thm
\label{T2}
Suppose that (\ref{m2.a}), (\ref{m2.b}), and (A2) are fulfilled.
Then the count process $(X_t)_{t\in\N_0}$ is absolutely regular
and the coefficients of absolute regularity satisfy
\begin{displaymath}
\beta^X(n) \,=\, O\big( \kappa^n \big),
\end{displaymath}
for any $\kappa>2\rho(\sqrt{B})$.
}
\bigskip

As for the Poisson autoregressive model, we also give a result for non-explosive covariates
with a less restrictive condition on the parameter space.
\bigskip

{\thm
\label{T2++}
Suppose that $\rho(B)<1$ and for $1\leq i\leq d$,
$\E\vert X_{0,i}\vert<\infty$ and $\sup_{t\in\N}\E\vert Z_{t,i}\vert<\infty$.
Then for any $\kappa>\rho(B)$,
\begin{displaymath}
\beta^X(n) \,=\, O\big( \kappa^n \big).
\end{displaymath}
}
\bigskip

\noindent
Proofs of Theorems~\ref{T1} to~\ref{T2++} are given in Subsections~\ref{SS4.1} to~\ref{SS4.4}, respectively.

%%%%%%%%%%%%%%%%%%%%%%%%%%%%%%%%%%%%%%%%%%%%%%%%%%%%%%%%%%%%%%%%%%%%%%%%%%%%%%%
\section{An application in statistics}
\label{S3}
%%%%%%%%%%%%%%%%%%%%%%%%%%%%%%%%%%%%%%%%%%%%%%%%%%%%%%%%%%%%%%%%%%%%%%%%%%%%%%%

In this section we apply our results to prove
asymptotic normality of a least squares estimator of the parameter in non-stationary 
Poisson-INARCH(1) and GINAR(1) models.
To simplify matters and in order to avoid any kind of high-level assumptions we focus
on the special cases where in (\ref{m1.b}) $A=0_{d\times d}$ and in (\ref{m1.b}) and (\ref{m2.c})
$B=\mbox{Diag}(b_1,\ldots,b_d)$.
This means that the parameter~$b_i$ can be estimated on the basis of the $i$th components
of the processes $(X_t)_{t\in\N_0}$ and $(Z_t)_{t\in\N_0}$ alone.
The simplification allows us to switch to the univariate case.

Suppose that the random variables $X_0,Z_1,X_1,Z_2,X_2,\ldots$ follow the model equations
\begin{equation}
\label{3.11}
X_t \,=\, b\circ X_{t-1} \,+\, Z_t,\qquad t\in\N,
\end{equation}
where $b\circ X_{t-1}$ can be represented as $\sum_{s=1}^{X_{t-1}} Y_{t,s}$.
$\{Y_{t,s}\colon\, t,s\in\N\}$ is a double array of independent and identically distributed
random variables such that $\{Y_{t,s}\colon\, s\in\N\}$ is also independent of $\F_{t-1}$.
We suppose that $Z_t$ follows a Poisson distribution with a non-random intensity~$\gamma_t$.
We have two special cases in mind:
If $Y_{t,s}\sim\mbox{Poi}(b)$, then $b\circ X_{t-1}\mid\F_{t-1}\sim \mbox{Poi}(bX_{t-1})$
and we obtain a Poisson-INARCH(1) model.
If $Y_{t,s}\sim\mbox{Bin}(1,b)$, then $b\circ X_{t-1}\mid\F_{t-1}\sim \mbox{Bin}(X_{t-1},b)$
and we obtain a Binomial-INARCH(1) or GINAR(1) model.
We shall assume that $0<b<1/4$ and that $X_0$ is non-random.
The trend of the count process $(X_t)_{t\in\N_0}$ is determined by that of $(Z_t)_{t\in\N_0}$, see Remark~\ref{R2} below.
We will then refer to this model as stochastic trend count autoregressive model, hereafter st-CAR.

\subsection{Asymptotic normality of a least squares estimator of an st-CAR model}

Before we proceed we impose some regularity conditions on the intensity process $(\gamma_t)_{t\in\N_0}$.
\bigskip
\begin{itemize}
\item[{\bf (A3)}\quad]
\begin{itemize}
\item[(i)\quad] The sequence $(\gamma_t)_{t\in\N_0}$ is monotonically non-decreasing and $\gamma_t\tinfty\infty$.
\item[(ii)\quad] $\gamma_{t+1}/\gamma_t\tinfty 1$ and there exists some $C<\infty$ such that
$\gamma_{2t}/\gamma_t\leq C \quad \forall t\in\N$.
\end{itemize}
\end{itemize}
\bigskip
Let $r_n=\sum_{t=1}^n \gamma_t^2$ and $s_n=\sum_{t=1}^n \gamma_t^3$.

{\rem
\label{R2}
We show below that the sequence $(\gamma_t)_{t\in\N_0}$ determines the growth rate of the count process
and that the sequences $(r_n)_{n\in\N}$ and $(s_n)_{n\in\N}$ together govern the rate of convergence
of a least squares estimator of the parameter~$b$.
Here are two examples: 
\begin{itemize}
\item[(i)\quad] (polynomial growth)\\
If $\gamma_t=d\,t^\alpha$, $d>0$ and $\alpha>0$, then we obtain from
$\int_0^n u^\beta\,du=\frac{n^{\beta+1}}{\beta+1}\leq \sum_{t=1}^n t^\beta
\leq \int_0^{n+1}u^\beta\,du=\frac{(n+1)^{\beta+1}}{\beta+1}$ ($\beta=2\alpha,3\alpha$) that
\begin{displaymath}
r_n \,=\, d^2\, \frac{n^{2\alpha+1}}{2\alpha+1} \, \big( 1 \,+\, o(1) \big) \qquad \mbox{ and } \qquad
s_n \,=\, d^3\, \frac{n^{3\alpha+1}}{3\alpha+1} \, \big( 1 \,+\, o(1) \big).
\end{displaymath}
\item[(ii)\quad] (logarithmic growth)\\
If $\gamma_t=d\,(\ln(t))^\alpha$, $d>0$ and $\alpha>0$, then the rates of growth of the sequences $(r_n)_{n\in\N}$ and
$(s_n)_{n\in\N}$ deviate from~$n$ only by logarithmic factors.
We have for $\beta=2\alpha,3\alpha$, 
\begin{displaymath}
\int_{1/n}^1 \Big( \frac{\ln(nx)}{\ln(n)} \Big)^\beta \, dx
\,\leq\, \frac{1}{n} \sum_{t=2}^n \Big( \frac{\ln(t)}{\ln(n)} \Big)^\beta
\,\leq\, \int_{1/n}^1 \Big( \frac{\ln(nx)}{\ln(n)} \Big)^\beta \, dx \,+\, \frac{1}{n}.
\end{displaymath}
Furthermore, it follows from dominated convergence that $\int_{1/n}^1 \big( \frac{\ln(nx)}{\ln(n)} \big)^\beta \, dx\ninfty 1$,
which implies that
\begin{displaymath}
\sum_{t=1}^n \big( \ln(t) \big)^\beta
\,=\, n\, \big( \ln(n) \big)^\beta \, \frac{1}{n} \sum_{t=2}^n \Big( \frac{\ln(t)}{\ln(n)} \Big)^\beta
\,=\, n\, \big( \ln(n) \big)^\beta \, \big( 1 \,+\, o(1) \big).
\end{displaymath}
Hence,
\begin{displaymath}
r_n \,=\, d^2\, n\, \big( \ln(n) \big)^{2\alpha} \, \big( 1 \,+\, o(1) \big) \qquad \mbox{ and } \qquad
s_n \,=\, d^3\, n\, \big( \ln(n) \big)^{3\alpha} \, \big( 1 \,+\, o(1) \big).
\end{displaymath}
\end{itemize}
}
\bigskip

Now we can draw a conclusion about the growth of the count variables~$X_t$.
Let $\varepsilon_t=b\circ X_{t-1}-b\,X_{t-1}$.
A repeated application of our model equation (\ref{3.11}) leads to
\begin{eqnarray}
\label{3.12a}
X_t & = & b\, X_{t-1} \,+\, \varepsilon_t \,+\, Z_t \nonumber \\
& = & b\, \big\{ b\, X_{t-2} \,+\, \varepsilon_{t-1} \,+\, Z_{t-1} \big\} \quad + \quad \varepsilon_t \,+\, Z_t \nonumber \\
& = & \ldots \,=\, b^t\, X_0 \,+\, \sum_{s=0}^{t-1} b^s\, \big( \varepsilon_{t-s} \,+\, Z_{t-s} \big) \nonumber \\
& = & b^t\, X_0 \,+\, \sum_{s=0}^{t-1} b^s\, \gamma_{t-s} \,+\, \sum_{s=0}^{t-1} b^s\, \big( \varepsilon_{t-s} \,+\, Z_{t-s} \,-\, \gamma_{t-s} \big).
\end{eqnarray}
Therefore,
\begin{eqnarray}
\label{3.12}
EX_t & = & b^t\, X_0 \,+\, \sum_{s=0}^{t-1} b^s\, \gamma_{t-s} \nonumber \\
& = & b^t\, X_0 \,+\, \gamma_t \Big\{ \frac{1-b^t}{1-b} \,-\, \sum_{s=0}^{t-1} b^s \big(1\,-\,\gamma_{t-s}/\gamma_t\big)
\Big\} \nonumber \\
& = & \frac{\gamma_t}{1-b} \, \big( 1 \,+\, o(1) \big),
\end{eqnarray}
that is, the growth of the $X_t$ is determined by the sequence $(\gamma_t)_{t\in\N_0}$.

Suppose now that $X_0,\ldots,X_n$ and $Z_1,\ldots,Z_n$ are observed.
Since
\begin{displaymath}
X_t \,-\, Z_t \,=\, b\, X_{t-1} \,+\, \varepsilon_t, \qquad t=1,\ldots,n,
\end{displaymath}
we can easily estimate the parameter~$b$.
The ordinary least squares estimator is given as
$\widehat{b}_n \,=\, \argmin_b \sum_{t=1}^n (X_t-Z_t \,-\, bX_{t-1})^2
\,=\, \sum_{t=1}^n X_{t-1} (X_t-Z_t)/\sum_{t=1}^n X_{t-1}^2$
and it follows that
\begin{displaymath}
\frac{r_n}{\sqrt{s_n}}\big( \widehat{b}_n \,-\, b \big)
\,=\, \frac{r_n}{\sum_{t=1}^n X_{t-1}^2} \, \sum_{t=1}^n X_{t-1} \varepsilon_t/\sqrt{s_n}.
\end{displaymath}
We analyze the terms $\sum_{t=1}^n X_{t-1}^2$ and $\sum_{t=1}^n X_{t-1} \varepsilon_t/\sqrt{s_n}$
separately.

We shall show in the course of the proof of Proposition~\ref{P3.1} that
\begin{subequations}
\begin{equation}
\label{3.13a}
\sum_{t=1}^n E\big[X_{t-1}^2\big] \,=\, \frac{r_n}{(1-b)^2}\, \big(1+o(1)\big)
\end{equation}
and
\begin{equation}
\label{3.13b}
\sum_{t=1}^n E\big[X_{t-1}^3\big] \,=\, \frac{s_n}{(1-b)^3}\, \big(1+o(1)\big).
\end{equation}
\end{subequations}
Using the mixing property stated in Theorems~\ref{T1} and~\ref{T2} together with an appropriate bound
for the moments of the~$X_t$ we shall also show that
\begin{subequations}
\begin{equation}
\label{3.14a}
\sum_{t=1}^n X_{t-1}^2 \,=\, \frac{r_n}{(1-b)^2} \,+\, o_P\big( r_n \big).
\end{equation}
and
\begin{equation}
\label{3.14b}
\sum_{t=1}^n X_{t-1}^3 \,=\, \frac{s_n}{(1-b)^3} \,+\, o_P\big( s_n \big).
\end{equation}
\end{subequations}
Let $Y_{n,t}=X_{t-1}\varepsilon_t/\sqrt{s_n}$. Then
\begin{equation}
\label{3.15a}
E\big( Y_{n,t} \mid \F_{t-1} \big) \,=\, 0
\end{equation}
and, using $\sum_{t=1}^n X_{t-1}^3/s_n \stackrel{P}{\longrightarrow} 1/(1-b)^3$,
\begin{equation}
\label{3.15b}
\sum_{t=1}^n E\big( Y_{n,t}^2 \mid \F_{t-1} \big)
\,=\, \sum_{t=1}^n X_{t-1}^2 \, E\big( \varepsilon_t^2 \mid \F_{t-1} \big) / s_n 
\,=\, \nu \sum_{t=1}^n X_{t-1}^3/s_n \,\stackrel{P}{\longrightarrow}\, \frac{\nu}{(1-b)^3},
\end{equation}
where $\nu=b$ in case of the Poisson-INARCH model and $\nu=b(1-b)$ in case of the Binomial-INARCH (i.e.~GINAR) model.
Finally, we shall show that
\begin{equation}
\label{3.15c}
\sum_{t=1}^n E\big[ Y_{n,t}^{8/3} \big] \,\ninfty\, 0,
\end{equation}
which implies, for $\epsilon>0$, that
\begin{displaymath}
\sum_{t=1}^n E\big[ Y_{n,t}^2 \, \1(|Z_{n,t}|>\epsilon ) \big]
\,\leq\, \sum_{t=1}^n E\big[ Y_{n,t}^{8/3} \big]/\epsilon^{2/3} \,\ninfty\, 0,
\end{displaymath}
i.e., the Lindeberg condition is satisfied.
Using (\ref{3.15a}) to (\ref{3.15c}) we obtain by a central limit theorem for sums of martingale differences
(see e.g. Corollary~3.8 in \citet{McL74} or Theorem~2 in conjunction with Lemma~2 in \citet{Bro71}) that
\begin{equation}
\label{3.16}
\frac{1}{\sqrt{s_n}} \, \sum_{t=1}^n X_{t-1} \varepsilon_t \,=\, \sum_{t=1}^n Y_{n,t}
\,\stackrel{d}{\longrightarrow}\, {\mathcal N}\big(0, \frac{\nu}{(1-b)^3}\big).
\end{equation}
(\ref{3.14a}) and (\ref{3.16}) lead to the following result.
\bigskip

{\prop
\label{P3.1}
Suppose that (\ref{3.11}) and the assumptions below this equation are satisfied.
Furthermore, suppose that (A3) is fulfilled. Then
\begin{displaymath}
\frac{r_n}{\sqrt{s_n}} \big( \widehat{b}_n \,-\,  b\big)
\,\stackrel{d}{\longrightarrow}\, {\mathcal N}\big( 0, \nu(1-b) \big).
\end{displaymath}
}
\bigskip

We see that the rate of convergence of~$\widehat{b}_n$ is $\sqrt{s_n}/r_n$.
In the special cases mentioned in the remark above, it is $n^{-(\alpha+1)/2}$ if $\gamma_t=dt^\alpha$
and $(n\,(\ln(n))^\alpha)^{-1/2}$ if $\gamma_t=d(\ln(t))^\alpha$.
In order to establish an asymptotic confidence interval for~$b$, it remains to estimate
the norming constants~$r_n$ and~$s_n$. (\ref{3.14a}), (\ref{3.14b}), and Proposition~\ref{P3.1} imply that
\begin{displaymath}
\big(1 \,-\, \widehat{b}_n\big)^2 \, \sum_{t=1}^n X_{t-1}^2 
\,=\, r_n \,+\, o_P(r_n)
\qquad {\rm and} \qquad
\big(1 \,-\, \widehat{b}_n\big)^3 \, \sum_{t=1}^n X_{t-1}^3 
\,=\, s_n \,+\, o_P(s_n).
\end{displaymath}
Therefore, an asymptotic confidence interval with a nominal coverage probability of $(1-\beta)$
is given by 
\begin{displaymath}
\Big[ \widehat{b}_n \,-\, \Phi^{-1}(1-\beta/2) \, K_n, \widehat{b}_n \,+\, \Phi^{-1}(1-\beta/2) \, K_n \Big],
\end{displaymath}
where $K_n=\sqrt{\widehat{b}_n} \sqrt{\sum_{t=1}^n X_{t-1}^3}/\sum_{t=1}^n X_{t-1}^2$ in case
of the Poisson-INARCH model and
$K_n=\sqrt{\widehat{b}_n(1-\widehat{b}_n)} \sqrt{\sum_{t=1}^n X_{t-1}^3}/\sum_{t=1}^n X_{t-1}^2$
for the Binomial-INARCH model.

\subsection{Numerical studies}
\label{subsuc::simulations}
We present a small numerical study to illustrate the asymp\-totic properties of the least squares estimator for an st-CAR model
provided by Proposition~\ref{P3.1}. To do so, we consider three values of parameter $b$, $0.1, 0.16$ and $0.23$.
For each value, we consider three different growth rates $\gamma_t=t$, $\gamma_t=t^2$ and $\gamma_t=\ln(t)$,
and simulate trajectories of length $n=50$, $n=100$ and $n=1,000$ for Poisson-INARCH (abbreviated as PINARCH)
and Binomial INARCH (abbreviated as BINARCH) models. We repeat this process $B=1,000$ times and compute
the least squares estimator of~$b$ and its average value (line LSE in Table~\ref{tb::simulations}).
Moreover, we also compute
the average value of~$K_n$ for Poisson-INARCH (line $K_n$-PINARCH in Table~\ref{tb::simulations}) and for
Binomial-INARCH (line $K_n$-BINARCH in Table~\ref{tb::simulations}) models. The simulation results look very promising,
even for moderate sample sizes $n=50$. As expected, the value of~$K_n$ is smaller for $\gamma_t = t^2$
than that for $\gamma_t = t$ which is itself smaller than that for  $\gamma_t = \ln(t)$, regardless of
the conditional distribution of the count process. 
Note that we also computed $K_n$ for both models in an incorrectly specified situation, i.e.~when one model is assumed to be true
but data were generated according to the other model. One can observe that assuming a Bernoulli or a Poisson distribution
in the thinning operator leads approximately to the same length for the confidence intervals, on average.
Our procedure seems then to be robust with respect to this particular model assumption.

\begin{table}[!ht]
	\footnotesize
	\begin{tabular}{lllccccccccc}
\cline{4-9}
&	&	& \multicolumn{3}{c}{PINARCH}			& \multicolumn{3}{c}{BINARCH}		\\
\cline{4-9}
&	& $\gamma_t$	& $t$ 	& $t^2$	& $ln~ t$	& $t$	& $t^2$	& $ln~ t$ \\ \hline 
\rowcolor[gray]{.9}n = 50	& $b= $ 0.1	& LSE	        & 0.0996	& 0.1000	& 0.0975	& 0.0998	& 0.0999	& 0.0989 \\

& 	& $K_n$-PINARCH	& 0.0090	& 0.0016	& 0.0262	& 0.0090	& 0.0016	& 0.0263 \\

& 	& $K_n$-BINARCH	& 0.0085	& 0.0015	& 0.0249	& 0.0085	& 0.0015	& 0.0250 \\

\rowcolor[gray]{.9}&  $b= $ 0.16	& LSE	        & 0.1597	& 0.1600	& 0.1567	& 0.1597	& 0.1600	& 0.1571 \\

& 	& $K_n$-PINARCH	& 0.0110	& 0.0019	& 0.0320	& 0.0110	& 0.0019	& 0.0321 \\

& 	& $K_n$-BINARCH	& 0.0101	& 0.0018	& 0.0293	& 0.0101	& 0.0018	& 0.0294 \\

\rowcolor[gray]{.9}&  $b= $ 0.23	& LSE	        & 0.2299	& 0.2300	& 0.2253	& 0.2295	& 0.2301	& 0.2289 \\

& 	& $K_n$-PINARCH	& 0.0126	& 0.0022	& 0.0365	& 0.0126	& 0.0022	& 0.0368 \\

& 	& $K_n$-BINARCH	& 0.0111	& 0.0020	& 0.0321	& 0.0111	& 0.0020	& 0.0322 \\
 \cline{2-9}
\rowcolor[gray]{.9}n = 100	&$b= $  0.1	& LSE	        & 0.100	        & 0.1000	& 0.0999	& 0.1000	& 0.1000	& 0.0994 \\

& 	& $K_n$-PINARCH	& 0.0045	& 0.0006	& 0.0170	& 0.0045	& 0.0006	& 0.0169 \\

& 	& $K_n$-BINARCH	& 0.0043	& 0.0005	& 0.0161	& 0.0043	& 0.0005	& 0.0160 \\

\rowcolor[gray]{.9}& $b= $ 0.16	& LSE	        & 0.1599	& 0.1600	& 0.1585	& 0.1600	& 0.1600	& 0.1592 \\

& 	& $K_n$-PINARCH	& 0.0055	& 0.0007	& 0.0206	& 0.0055	& 0.0007	& 0.0206 \\

& 	& $K_n$-BINARCH	& 0.0050	& 0.0006	& 0.0189	& 0.0050	& 0.0006	& 0.0189 \\

 \rowcolor[gray]{.9}&  $b= $ 0.23 	& LSE	        & 0.2298	& 0.2300	& 0.2276	& 0.2300	& 0.2300	& 0.2292 \\

& 	& $K_n$-PINARCH	& 0.0063	& 0.0008	& 0.0235	& 0.0063	& 0.0008	& 0.0235 \\

& 	& $K_n$-BINARCH	& 0.0055	& 0.0007	& 0.0207	& 0.0055	& 0.0007	& 0.0206 \\
  \cline{2-9}
\rowcolor[gray]{.9}n = 1000& $b= $ 0.1	& 	LSE	& 0.1000	& 0.1000	& 0.1002	& 0.1000	& 0.1000	& 0.0999 \\

& 	& $K_n$-PINARCH	& 0.0005	& 0.0000	& 0.0041	& 0.0005	& 0.0000	& 0.0041 \\

& 	& $K_n$-BINARCH	& 0.0004	& 0.0000	& 0.0039	& 0.0004	& 0.0000	& 0.0039 \\

\rowcolor[gray]{.9}&  $b= $ 0.16	& LSE	        & 0.1600	& 0.1600	& 0.1598	& 0.1600	& 0.1600	& 0.1600 \\

& 	& $K_n$-PINARCH	& 0.0006	& 0.0000	& 0.0050	& 0.0006	& 0.0000	& 0.0050 \\

& 	& $K_n$-BINARCH	& 0.0005	& 0.0000	& 0.0046	& 0.0005	& 0.0000	& 0.0046 \\

\rowcolor[gray]{.9}&  $b= $ 0.23	& LSE	        & 0.2300	& 0.2300	& 0.2298	& 0.2300	& 0.2300	& 0.2300 \\

& 	& $K_n$-PINARCH	& 0.0006	& 0.0000	& 0.0058	& 0.0006	& 0.0000	& 0.0058 \\

& 	& $K_n$-BINARCH	& 0.0006	& 0.0000	& 0.0051	& 0.0006	& 0.0000	& 0.0051 \\ \hline
	\end{tabular}
\caption{Simulation results of least squares estimator for st-CAR model}\label{tb::simulations}
\end{table}

\subsection{Real data example}
\label{sec::real_data_analysis}
As an illustration, we analyse the count number of questions about  scrapy, a python framework for web scraping.
We consider the number of questions about natural language processing (NLP) as a predictor. Actually, NLP entered
a new era as of  2010 with recent development on neural network models.  BERT (Bidirectional Encoder Representations
from Transformers) is the state-of-art model for NLP developed by Google in 2018. There are more and more NLP projects
based on internet content, like sentiment analysis,  due to the high use of social media,  among others.
For data harvesting and processing purposes, many libraries such as  Requests or  BeautifulSoup came up in python.
Beside these ones, scrapy represents a whole framework, meaning it comes with a set of rules and conventions, that
allow us to  efficiently  solve common web scraping  problems. 

For this small data analysis study, we download  data  of monthly number of various questions about NLP and scrapy on stackoverflow, 
the largest online community for programmers to learn and  share their knowledge.
The data are available on \url{https://www.kaggle.com/datasets/aishu200023/stackindex}.
They were collected between January 2009 and December 2019 (see Fig.~\ref{fig::questions}).
Estimated value of parameter $b$ is $0.23269$ with $K_n-$INGARCH $= 0.00453$ and $K_n-$BINARCH $= 0.00397.$

Let us mention a possible application of our results to the derivation of an upper bound for the conditional MSE of prediction.
For instance, in the Poisson case, suppose that we want to bound
\begin{displaymath}
S_n \,:=\, E\big[(X_{n+1}-\widehat{X}_{n+1})^2\big| \mathcal{F}_n\big] \,=\, X_n b \,+\, \left(\widehat{b}_n-b\right)^2X_n^2,
\end{displaymath}
where $\widehat{X}_{n+1}=\widehat{b}_n X_n+Z_{n+1}$. For a given level $\beta$, let us define 
$q_\beta=\Phi\left(1-\beta/2\right)K_n$. We then get
\begin{displaymath}
P\left(S_n\leq X_n\widehat{b}_n+q_\beta X_n+q_\beta^2 X_n^2\right)
\,\geq\, P\left(\left\vert \widehat{b}_n-b\right\vert\leq q_\beta\right)
\end{displaymath}
and we deduce that 
\begin{displaymath}
\liminf_{n\rightarrow \infty} P\left(S_n\leq X_n\widehat{b}_n+q_\beta X_n+q_\beta^2 X_n^2\right) \,\geq\, 1-\beta.
\end{displaymath}

\begin{figure}[!h]
\includegraphics[scale=0.55]{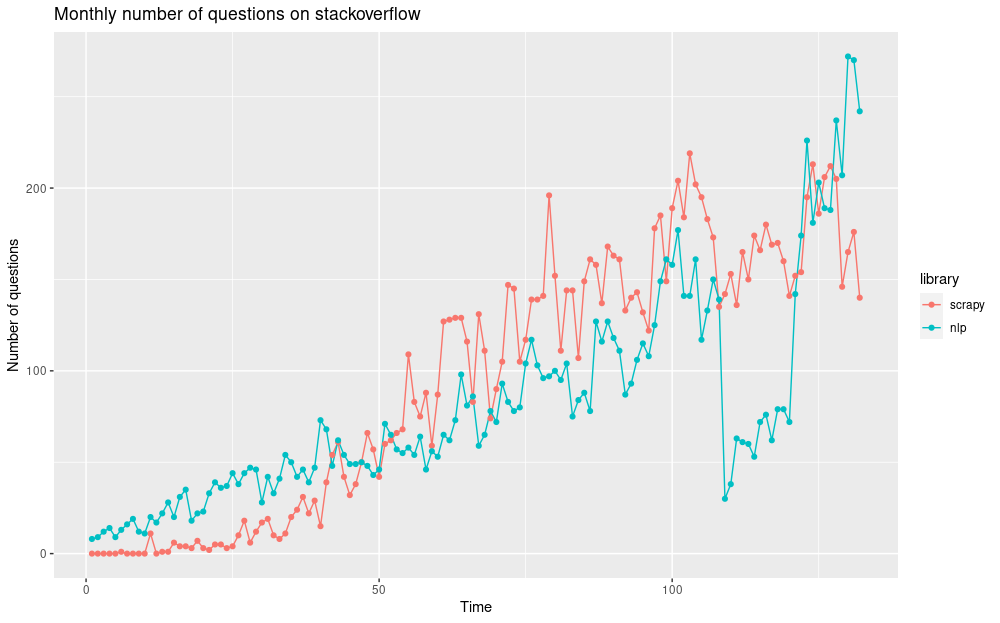}
	\caption{Number of questions asked about NLP (blue line) and scrapy (red line) on stackoverflow}\label{fig::questions}
\end{figure}

%%%%%%%%%%%%%%%%%%%%%%%%%%%%%%%%%%%%%%%%%%%%%%%%

%%%%%%%%%%%%%%%%%%%%%%%%%%%%%%%%%%%%%%%%%%%%%%%%%%%%%%%%%%%%%%%%%%%%%%%%%%%%%%%
\section{Proofs of the main results}
\label{S4}
%%%%%%%%%%%%%%%%%%%%%%%%%%%%%%%%%%%%%%%%%%%%%%%%%%%%%%%%%%%%%%%%%%%%%%%%%%%%%%%

\subsection{Proof of Theorem~\ref{T1}}
\label{SS4.1}

The proof of Theorem~\ref{T1} is based on the following two lemmas.

{\lem
\label{LA.1}
Suppose that (\ref{m1.a}), (\ref{m1.b}), and (A1) are fulfilled.
Let $(\wtl_{k})_{k\in\N_0}$ and $(\wtl_{k}')_{k\in\N_0}$ be two independent copies
of the intensity process. Then
\begin{displaymath}
\sup_{k\in\N} \widetilde{E} \big\| \sqrt{\wtl_k} \,-\, \sqrt{\wtl_k'} \big\|_1 \,<\, \infty.
\end{displaymath}
}

\begin{proof}[Proof of Lemma~\ref{LA.1}]
Recall that
$\lambda_{t,i}=\sum_{j=1}^d A_{ij} \lambda_{t-1,j} + \sum_{j=1}^d B_{ij} X_{t-1,j} + Z_{t-1,i}$.
For $t\in\N$, $i\in\{1,\ldots,d\}$, we split up
\begin{eqnarray}
\label{pla1.1}
\lefteqn{ \Big| \sqrt{\wtl_{t,i}} \,-\, \sqrt{\wtl_{t,i}'} \Big| } \nonumber \\
& \leq & \sum_{j=1}^d \Big| \sqrt{ (A_{ij}+B_{ij})\; \wtl_{t-1,j} } \,-\, \sqrt{ (A_{ij}+B_{ij})\; \wtl_{t-1,j}' } \Big| \nonumber \\
& & {} \,+\, \sum_{j=1}^d \bigg\{ \Big| \sqrt{ A_{ij}\wtl_{t-1,j} \,+\, B_{ij} \wtx_{t-1,j} }
\,-\, \sqrt{ (A_{ij}+B_{ij})\; \wtl_{t-1,j} } \Big| \nonumber \\
& & \qquad \qquad \qquad
\,+\, \Big| \sqrt{ A_{ij}\wtl_{t-1,j}' \,+\, B_{ij} \wtx_{t-1,j}' } \,-\, \sqrt{ (A_{ij}+B_{ij})\; \wtl_{t-1,j}' } \Big| \bigg\}
\nonumber \\
& & {} \,+\, \Big| \sqrt{\wtz_{t-1,i}} \,-\, \sqrt{\wtz_{t-1,i}'} \Big| \nonumber \\
& \leq & \sum_{j=1}^d \big(\sqrt{A_{ij}+B_{ij}}\big) \; \Big| \sqrt{\wtl_{t-1,j}} \,-\, \sqrt{\wtl_{t-1,j}'} \Big| \nonumber \\
& & {} \,+\, \sum_{j=1}^d \sqrt{B_{ij}} \; \Big( \big|\sqrt{\wtx_{t-1,j}}-\sqrt{\wtl_{t-1,j}}\big|
\,+\, \big|\sqrt{\wtx_{t-1,j}'}-\sqrt{\wtl_{t-1,j}'}\big| \Big) \nonumber \\
& & {} \,+\, \Big| \sqrt{\wtz_{t-1,i}} \,-\, \sqrt{\wtz_{t-1,i}'} \Big| \nonumber \\
& =: & R_{t,1}^{(i)} \,+\, R_{t,2}^{(i)} \,+\, R_{t,3}^{(i)},
\end{eqnarray}
say. 
Since, for $X\sim\mbox{Poi}(\lambda)$,
$E\big|\sqrt{X}-\sqrt{\lambda}\big|\leq E|X-\lambda|/\sqrt{\lambda}\leq \sqrt{E(X-\lambda)^2/\lambda}=1$,
we obtain
\begin{eqnarray}
\label{pla1.3}
\lefteqn{ \widetilde{E}\Big( R_{t,2}^{(i)} \;\Big|\; \wtl_{t-1}, \wtl_{t-1}' \Big) } \nonumber \\
& \leq & \sum_{j=1}^d \sqrt{B_{ij}}
\;\; \bigg\{ \widetilde{E}\Big( \big| \sqrt{\wtx_{t-1,j}} - \sqrt{\wtl_{t-1,j}} \big| \;\Big|\; \wtl_{t-1}, \wtl_{t-1}' \Big)
\,+\, \widetilde{E}\Big( \big| \sqrt{\wtx_{t-1,j}'} - \sqrt{\wtl_{t-1,j}'} \big| \;\Big|\; \wtl_{t-1}, \wtl_{t-1}' \Big) \bigg\} \nonumber \\
& \leq & 2 \; \sum_{j=1}^d \sqrt{ B_{ij} }.
\end{eqnarray}
Finally, we have
\begin{equation}
\label{pla1.4}
\widetilde{E}\Big( R_{t,3}^{(i)} \;\big|\; \wtz_{t-1}, \wtz_{t-1}' \Big)
\,=\, \widetilde{E} \Big| \sqrt{\wtz_{t-1,i}} \,-\, \sqrt{\wtz_{t-1,i}'} \Big|
\,\leq\, 2 \; E\Big| \sqrt{Z_{t-1,i}} \,-\, E\sqrt{Z_{t-1,i}} \Big|.
\end{equation}
It follows from (\ref{pla1.1}) to (\ref{pla1.4}) that for $1\leq i\leq d$
\begin{displaymath}
\E\big| \sqrt{\widetilde{\lambda}_{t,i}}-\sqrt{\widetilde{\lambda}'_{t,i}} \big|
\,\leq\, \sum_{j=1}^d C_{ij}
\E\big| \sqrt{\widetilde{\lambda}_{t-1,j}}-\sqrt{\widetilde{\lambda}'_{t-1,j}} \big| \,+\, b_i,
\end{displaymath}
where $C_{ij}=\sqrt{A_{ij}+B_{ij}}$ and
$b_i=2\sum_{j=1}^d\sqrt{B_{ij}}+2\sup_{t\geq 0} E\big| \sqrt{Z_{t,i}} \,-\, E\sqrt{Z_{t,i}} \big|$.
Since $\rho(C)<1$, an application of Lemma~\ref{spectral} yields the result.
\end{proof}
\bigskip

{\lem
\label{LA.2}
Suppose that (\ref{m1.a}), (\ref{m1.b}), and (A1) are fulfilled.
Then there exists a coupling such that the following results hold true.
\begin{itemize}
\item[(i)\quad]
For $n\geq 1$,
\begin{displaymath}
\widetilde{E} \big\| \sqrt{\wtl_{k+n}} \,-\, \sqrt{\wtl_{k+n}'} \big\|_1
\,\leq\, \big\| (\sqrt{A}+2\sqrt{B})^{n-1} \big\|_1 \; \widetilde{E} \big\| \sqrt{\wtl_{k+1}} \,-\, \sqrt{\wtl_{k+1}'} \big\|_1
\end{displaymath}
and 
\begin{displaymath}
\widetilde{P}\big( \wtx_{k+n}\neq \wtx_{k+n}' \big)
\,=\, O\big( \big\| (\sqrt{A}+2\sqrt{B})^{n-1} \big\|_1 \; \widetilde{E} \big\| \sqrt{\wtl_{k+1}} \,-\, \sqrt{\wtl_{k+1}'} \big\|_1 \big).
\end{displaymath}
\item[(ii)\quad]
For $n,r\geq 1$,
\begin{eqnarray*}
\widetilde{E} \Big[ \big\| \sqrt{\wtl_{k+n+r}} \,-\, \sqrt{\wtl_{k+n+r}'} \big\|_1 \;
\1\big( \wtx_{k+n}=\wtx_{k+n}', \ldots, \wtx_{k+n+r-1}=\wtx_{k+n+r-1}' \big) \Big] \\
\leq \,  \big\| (\sqrt{A})^r \big\|_1 \; \big\| (\sqrt{A}+2\sqrt{B})^{n-1} \big\|_1 \;
\widetilde{E} \big\| \sqrt{\wtl_{k+1}} \,-\, \sqrt{\wtl_{k+1}'} \big\|_1
\end{eqnarray*}
and
\begin{eqnarray*}
\widetilde{P}\big( \wtx_{k+n+r}\neq \wtx_{k+n+r}', \wtx_{k+n}=\wtx_{k+n}', \ldots, \wtx_{k+n+r-1}=\wtx_{k+n+r-1}' \big) \\
=\, O\big( \big\| (\sqrt{A})^r \big\|_1 \; \big\| (\sqrt{A}+2\sqrt{B})^{n-1} \big\|_1
\; \widetilde{E} \big\| \sqrt{\wtl_{k+1}} \,-\, \sqrt{\wtl_{k+1}'} \big\|_1 \big).
\end{eqnarray*}
\end{itemize}
}
\bigskip

\begin{proof}[Proof of Lemma~\ref{LA.2}]
$ $\\
\mbox{(i)\quad}
Let $k+1\leq t\leq k+n-1$. Given $\wtl_t$ and $\wtl_t'$, we apply a maximal coupling for $\wtx_{t,i}$ and $\wtx_{t,i}'$
as described in Lemma~\ref{couplinglemma} below.
Furthermore, we couple $\wtz_t$ and $\wtz_t'$ such that they are equal.
Note that, for $\lambda>\lambda'\geq 0$, $X\sim\mbox{Poi}(\lambda)$ is stochastically larger than $X'\sim\mbox{Poi}(\lambda')$.
Therefore, it follows that
$\wtx_{t,j}\geq \wtx_{t,j}'$ if $\wtl_{t,j}\geq \wtl_{t,j}'$
and, vice versa, $\wtx_{t,j}\leq \wtx_{t,j}'$ if $\wtl_{t,j}\leq \wtl_{t,j}'$; see c) of Lemma~\ref{couplinglemma} below.
This allows us to apply Lemma~\ref{L.inequality} and we obtain
\begin{eqnarray*}
\lefteqn{ \widetilde{E} \Big( \big| \sqrt{ A_{ij}\wtl_{t,j} \,+\, B_{ij}\wtx_{t,j} }
\,-\, \sqrt{ A_{ij}\wtl_{t,j}' \,+\, B_{ij}\wtx_{t,j}' } \big| \;\Big|\; \wtl_t, \wtl_t', \; \wtl_{t,j}\geq \wtl_{t,j}' \Big) } \\
& = & \widetilde{E} \Big( \sqrt{ A_{ij}\wtl_{t,j} \,+\, B_{ij}\wtx_{t,j} }
\,-\, \sqrt{ A_{ij}\wtl_{t,j}' \,+\, B_{ij}\wtx_{t,j}' } \;\Big|\; \wtl_t, \wtl_t', \; \wtl_{t,j}\geq \wtl_{t,j}' \Big) \\
& \leq & \sqrt{A_{ij}} \; \big( \sqrt{\wtl_{t,j}} \,-\, \sqrt{\wtl_{t,j}'} \big)
\,+\, \sqrt{B_{ij}} \; \widetilde{E}\Big( \sqrt{\wtx_{i,j}} \,-\, \sqrt{\wtx_{i,j}'}
\;\Big|\; \wtl_t, \wtl_t', \; \wtl_{t,j}\geq \wtl_{t,j}' \Big) \\
& \leq & \big(\sqrt{ A_{ij} } + 2\sqrt{ B_{ij} }\big) \; \Big| \sqrt{\wtl_{t,j}} \,-\, \sqrt{\wtl_{t,j}'} \Big|
\end{eqnarray*}
and, similarly, 
\begin{eqnarray*}
\lefteqn{ \widetilde{E} \Big( \big| \sqrt{ A_{ij}\wtx_{t,j} \,+\, B_{ij}\wtl_{t,j} }
\,-\, \sqrt{ A_{ij}\wtx_{t,j}' \,+\, B_{ij}\wtl_{t,j}' } \big| \;\Big|\; \wtl_t, \wtl_t', \; \wtl_{t,j}\leq \wtl_{t,j}' \Big) } \\
& = & \widetilde{E} \Big( \sqrt{ A_{ij}\wtl_{t,j}' \,+\, B_{ij}\wtx_{t,j}' }
\,-\, \sqrt{ A_{ij}\wtl_{t,j} \,+\, B_{ij}\wtx_{t,j} } \;\Big|\; \wtl_t, \wtl_t', \; \wtl_{t,j}\leq \wtl_{t,j}' \Big) \\
& \leq & \big(\sqrt{ A_{ij} } + 2\sqrt{ B_{ij} }\big) \; \Big| \sqrt{\wtl_{t,j}} \,-\, \sqrt{\wtl_{t,j}'} \Big|.
\end{eqnarray*}

Therefore, and since $\wtz_t=\wtz_t'$ we obtain that
\begin{eqnarray*}
\widetilde{E} \Big( \big| \sqrt{\wtl_{t+1,i}} \,-\, \sqrt{\wtl_{t+1,i}'} \big| \;\Big|\; \wtl_t, \wtl_t' \Big)
& \leq & \sum_{j=1}^d \widetilde{E} \Big( \big| \sqrt{ A_{ij}\wtl_{t,j} \,+\, B_{ij}\wtx_{t,j} }
\,-\, \sqrt{ A_{ij}\wtl_{t,j}' \,+\, B_{ij}\wtx_{t,j}' } \big| \;\Big|\; \wtl_t, \wtl_t' \Big) \\
& \leq & \sum_{j=1}^d \big( \sqrt{ A_{ij} } + 2\sqrt{ B_{ij} } \big) \; \Big| \sqrt{\wtl_{t,j}} \,-\, \sqrt{\wtl_{t,j}'} \Big|.
\end{eqnarray*}
Taking expectation on both sides of this inequality we see that 
\begin{displaymath}
\widetilde{E} \big|  \sqrt{\wtl_{t+1,i}} \,-\, \sqrt{\wtl_{t+1,i}'} \big|
\,\leq\, \sum_{j=1}^d \big( \sqrt{ A_{ij} } + 2\sqrt{ B_{ij} } \big) \; \widetilde{E}\big| \sqrt{\wtl_{t,j}} \,-\, \sqrt{\wtl_{t,j}'} \big|.
\end{displaymath}
{}From Lemma~\ref{spectral} with $b_i=0$, we obtain 
\begin{displaymath}
\widetilde{E} \big\| \sqrt{\wtl_{k+n}} \,-\, \sqrt{\wtl_{k+n}'} \big\|_1
\,\leq\, \big\| (\sqrt{A} \,+\, 2\sqrt{B})^{n-1} \big\|_1 \; \widetilde{E} \big\| \sqrt{\wtl_{k+1}} \,-\, \sqrt{\wtl_{k+1}'} \big\|_1.
\end{displaymath}
\mbox{(ii)\quad}
For $t>k+n$ and $1\leq i\leq d$, we have on the event $H:=\left\{\wtx_{k+n}=\wtx_{k+n}',\ldots,\wtx_{t-1}=\wtx_{t-1}\right\}$,
\begin{displaymath}
\big| \sqrt{\wtl_{t,i}} \,-\, \sqrt{\wtl_{t,i}'} \big | \;
\,\leq\, \sum_{j=1}^d (\sqrt{A})_{ij} \; \big| \sqrt{\wtl_{t-1,j}} \,-\, \sqrt{\wtl_{t-1,j}'} \big |.
\end{displaymath}
Iterating the previous bound, as in Lemma \ref{spectral}, we get 
\begin{displaymath}
\big\Vert \sqrt{\widetilde{\lambda}_t}-\sqrt{\widetilde{\lambda}'_t}\big\Vert_1
\,\leq\, \big\Vert (\sqrt{A})^{t-(k+n)}\big\Vert_1
\big\Vert \sqrt{\widetilde{\lambda}_{k+n}}- \sqrt{\widetilde{\lambda}'_{k+n}}\big\Vert_1
\end{displaymath}
on the event $H$.

In both cases (i) and (ii), 
the upper estimates for the probabilities $\widetilde{P}\big( \wtx_{k+n}\neq \wtx_{k+n}' \big)$ and 
$\widetilde{P}\big( \wtx_{k+n+r}\neq \wtx_{k+n+r}', \wtx_{k+n}=\wtx_{k+n}', \ldots, \wtx_{k+n+r-1}=\wtx_{k+n+r-1}' \big)$
follow from the maximal coupling and the fact that
$d_{TV}\big(\mbox{Poi}(\lambda),\mbox{Poi}(\lambda')\big)\leq (2/e) \big|\sqrt{\lambda}-\sqrt{\lambda'}\big|$.
\end{proof}
\bigskip

\noindent
Now we turn to the proof of our first major result.
\bigskip

\begin{proof}[Proof of Theorem~\ref{T1}]
Let $k\in\N_0$, and let $(\wtx_t)_{t\in\N_0}$ and $(\wtx_t')_{t\in\N_0}$ be two versions of the count process,
where $(\wtx_0,\ldots,\wtx_k)$ and $(\wtx_0',\ldots,\wtx_k')$ are independent,
and where $\wtx_{k+1},\wtx_{k+2},\ldots$ are coupled with their respective counterparts $\wtx_{k+1}',\wtx_{k+2}',\ldots$
as described in the proof of Lemma~\ref{LA.2}.
Then we obtain from (\ref{1.2}), Lemma~\ref{LA.1}, Lemma~\ref{LA.2}, and (\ref{l1-sr}) that
\begin{eqnarray*}
\beta^X(k,n)
& \leq & \widetilde{P}\left( \wtx_{k+n+r} \neq \wtx_{k+n+r}' \quad \mbox{for some } r\in\N_0 \right) \\
& \leq & \widetilde{P}\left( \wtx_{k+n} \neq \wtx_{k+n}' \right) \\
& & {} \,+\, \sum_{r=1}^\infty \widetilde{P}\left( \wtx_{k+n+r} \neq \wtx_{k+n+r}',
\wtx_{k+n}=\wtx_{k+n}',\ldots,\wtx_{k+n+r-1}=\wtx_{k+n+r-1}' \right) \\
& = & O\left( \kappa^n \right).
\end{eqnarray*}
\end{proof}

\subsection{Proof of Theorem~\ref{T2+}}
\label{SS4.2}

\begin{proof}
Let~$k\in\N$. To define our coupling we consider
a sequence of i.i.d.~random variables $\widetilde{J}:=\big(N^{(t)},\widetilde{Z}_t\big)_{t\geq 0}$ such that 
$N^{(t)}=\big(N^{(t,1)},\ldots,N^{(t,d)}\big)$ is a multivariate point process taking values in $\N_0^d$  
and such that, for $1\leq i\leq d$, $N^{(t,i)}=\big(N^{(t,i)}_u\big)_{u\geq 0}$ is a Poisson process with intensity~$1$. 
Moreover, $\widetilde{Z}_t$ is a copy of $Z_t$.
For $\lambda\in \R_+^d$, we set $N^{(t)}_{\lambda}=\big(N^{(t,1)}_{\lambda_1},\ldots,N^{(t,d)}_{\lambda_d}\big)$.
We also consider an independent copy $\widetilde{J}':=\big(\bar{N}^{(t)},\widetilde{Z}'_t\big)_{t\geq 0}$ of $\widetilde{J}$.
Finally, we consider two independent copies $\widetilde{\lambda}_0$ and $\widetilde{\lambda}'_0$ of $\lambda_0$,
independent of the pair $\left(\widetilde{J},\widetilde{J}'\right)$ and, for $0\leq t\leq k$, we set 
$\widetilde{X}_t=N^{(t)}_{\widetilde{\lambda}_t}$, $\widetilde{X}_t'=\bar{N}^{(t)}_{\widetilde{\lambda}'_t}$.
Then $\widetilde{X}_t\mid\widetilde{\lambda}_t\sim\mbox{Poi}_{dep}(\widetilde{\lambda}_t)$
and $\widetilde{X}_t'\mid\widetilde{\lambda}_t'\sim\mbox{Poi}_{dep}(\widetilde{\lambda}_t')$.
Furthermore, for $1\leq t\leq k$,
\begin{displaymath}
\widetilde{\lambda}_t \,=\, A \widetilde{\lambda}_{t-1} \,+\, B \widetilde{X}_t+\widetilde{Z}_{t-1},
\qquad \widetilde{\lambda}'_t \,=\, A\widetilde{\lambda}'_{t-1} \,+\, B\widetilde{X}_t'+\widetilde{Z}'_{t-1}.
\end{displaymath}
For $t\geq k+1$, we set $\widetilde{X}_t=N^{(t)}_{\widetilde{\lambda}_t}$ and $\widetilde{X}'_t=N^{(t)}_{\widetilde{\lambda}'_t}$ with the same recursion formula for $\widetilde{\lambda}_t$ and replacing $\widetilde{Z}_{t-1}$ by $\widetilde{Z}_{t-1}'$ in the recursive formula for  $\widetilde{\lambda}'_t$. We have, for any $t\geq 1$ and $1\leq i\leq d$,
\begin{displaymath} 
v_{t,i} \,:=\, \E\widetilde{\lambda}_{t,i}\leq \E Z_{t-1,i} +\\ \sum_{j=1}^d \left(A_{ij}+B_{ij}\right)v_{t-1,j}.
\end{displaymath}
Setting $b_i=\sup_{t\geq 0}\E Z_{t,i}$, we can apply Lemma \ref{spectral} which will ensure that
$\sup_{t\geq 1}\Vert v_t\Vert_1<\infty$.
The same property holds true if $v_t$ is replaced with $v_t'$ with $v_{t,i}'=\E\widetilde{\lambda}_{t,i}'$, $1\leq i\leq d$.
We then get 
\begin{displaymath}
\sup_{t\geq 0}\E\Vert \widetilde{\lambda}_t-\widetilde{\lambda}_t'\Vert_1 \,<\, \infty. 
\end{displaymath}
Moreover for $t\geq k+1$, we have
\begin{displaymath}
\P\left(\widetilde{X}_t\neq \widetilde{X}_t'\right)
\,\leq\, \sum_{i=1}^d\P\left(\widetilde{X}_{t,i}\neq \widetilde{X}'_{t,i}\right)
\,\leq\, \sum_{i=1}^d\E\left \vert \widetilde{X}_{t,i}-\widetilde{X}'_{t,i}\right\vert 
\,=\, \sum_{i=1}^d \E\left\vert \widetilde{\lambda}_{t,i}-\widetilde{\lambda}'_{t,i}\right\vert.
\end{displaymath}
Setting $w_{t,i}=\E\left\vert \widetilde{\lambda}_{t,i}-\widetilde{\lambda}'_{t,i}\right\vert$
for $1\leq i\leq d$ and $t\geq k+1$, we have
\begin{displaymath}
w_{t,i} \,\leq\, \sum_{j=1}^d \left(A_{ij}+B_{ij}\right)w_{t-1,j}.
\end{displaymath}
In matrix notation we obtain, with $w_t=(w_{t,1},\ldots,w_{t,n})^T$,
\begin{displaymath}
\| w_{k+n+r} \|_1 \,\leq\, \big\| (A+B)^{n+r-1} \big\|_1 \|w_{k+1}\|_1,
\end{displaymath}
which implies, in conjunction with (\ref{l1-sr}),
\begin{displaymath}
\beta^{X}(k,n) \,\leq\, \sum_{r=0}^\infty \P\left(\widetilde{X}_{k+n+r}\neq \widetilde{X}_{k+n+r}'\right)
\,=\, O\big( \kappa^n \big).
\end{displaymath}
\end{proof}

\subsection{Proof of Theorem~\ref{T2}}
\label{SS4.3}

The proof of Theorem~\ref{T2} is based on the following two lemmas.

{\lem
\label{LA.3}
Suppose that (\ref{m2.a}), (\ref{m2.b}), and (A2) are fulfilled.
Let $(\wtx_k)_{k\in\N_0}$ and $(\wtx_k)_{k\in\N_0}$ be two independent copies of the count process. Then
\begin{displaymath}
\sup_{k\in\N_0} \widetilde{E} \big\| \sqrt{\wtx_k} \,-\, \sqrt{\wtx_k'} \big\|_1 \,<\, \infty.
\end{displaymath}
}

\begin{proof}[Proof of Lemma~\ref{LA.3}]
Recall that $X_{t,i}=\sum_{j=1}^d Y_t^{i,j}+Z_{t,i}$, where $Y_t^{i,j}=\sum_{s=1}^{X_{t-1,j}} Y_{t,s}^{i,j}$.
Furthermore, let $\lambda_t^{i,j}=E( Y_t^{i,j} \mid X_{t-1} ) =B_{ij} X_{t-1,j}$.
Then
\begin{eqnarray}
\label{pla3.1}
\Big| \sqrt{ \wtx_{t,i} } \,-\, \sqrt{ \wtx_{t,i}' } \Big|
& \leq & \Big| \sqrt{ \sum_{j=1}^d \wtl_t^{i,j} } \,-\, \sqrt{ \sum_{j=1}^d \wtl_t^{i,j'} } \Big| \nonumber \\
& & {} \,+\, \Big| \sqrt{ \sum_{j=1}^d \wty_t^{i,j} } \,-\, \sqrt{ \sum_{j=1}^d \wtl_t^{i,j} } \Big| 
\;+\; \Big| \sqrt{ \sum_{j=1}^d \wty_t^{i,j'} } \,-\, \sqrt{ \sum_{j=1}^d \wtl_t^{i,j'} } \Big| \nonumber \\
& & {} \,+\, \Big| \sqrt{\wtz_{t,i}} \,-\, \sqrt{\wtz_{t,i}} \Big| \nonumber \\
& \leq & \sum_{j=1}^d \Big| \sqrt{ \wtl_t^{i,j} } \,-\, \sqrt{ \wtl_t^{i,j'} } \Big| \nonumber \\
& & {} \,+\, \sum_{j=1}^d \bigg\{ \big| \sqrt{\wty_t^{i,j}} \,-\, \sqrt{\wtl_t^{i,j}} \big| 
\;+\; \big| \sqrt{\wty_t^{i,j'}} \,-\, \sqrt{\wtl_t^{i,j'}} \big| \bigg\} \nonumber \\
& & {} \,+\, \Big| \sqrt{\wtz_{t,i}} \,-\, \sqrt{\wtz_{t,i}} \Big| \nonumber \\
& = & S_{t,i}^{(1)} \,+\, S_{t,i}^{(2)} \,+\, S_{t,i}^{(3)},
\end{eqnarray}
say.

We have
\begin{equation}
\label{pla3.2}
\widetilde{E} \Big( S_{t,i}^{(1)} \;\Big|\; \wtx_{t-1},\wtx_{t-1}' \Big)
\,=\, \sum_{j=1}^d \Big| \sqrt{ \wtl_t^{i,j} } \,-\, \sqrt{ \wtl_t^{i,j'} } \Big|
\,=\, \sum_{j=1}^d \sqrt{B_{ij}} \; \Big| \sqrt{\wtx_{t-1,j}} \,-\, \sqrt{\wtx_{t-1,j}'} \Big|.
\end{equation}
Since, for $Y\sim\mbox{Bin}(n,p)$, $E|\sqrt{Y}-\sqrt{np}|\leq E|Y-np|/\sqrt{np}\leq \sqrt{E(Y-np)^2/np}\leq 1$
we obtain that 
\begin{eqnarray}
\label{pla3.3}
\lefteqn{ \widetilde{E} \big( S_{t,i}^{(2)} \;\mid\; \wtx_{t-1}, \wtx_{t-1}' \big) } \nonumber \\
& \leq & \sum_{j=1}^d \bigg\{ \widetilde{E}\Big( \big| \sqrt{\wty_t^{i,j}} - \sqrt{\wtl_t^{i,j}} \big| \;\Big|\; \wtx_{t-1}, \wtx_{t-1}' \Big)
\,+\, \widetilde{E}\Big( \big| \sqrt{\wty_t^{i,j'}} - \sqrt{\wtl_t^{i,j'}} \big| \;\Big|\; \wtx_{t-1}, \wtx_{t-1}' \Big) \bigg\} \nonumber \\
& \leq & 2\, d.
\end{eqnarray}
Finally, we have
\begin{equation}
\label{pla3.4}
\widetilde{E}\Big( S_{t,i}^{(3)} \;\big|\; \wtz_t, \wtz_t' \Big)
\,=\, \widetilde{E} \Big| \sqrt{\wtz_{t,i}} \,-\, \sqrt{\wtz_{t,i}'} \Big|
\,\leq\, 2 \; E\Big| \sqrt{Z_{t,i}} \,-\, E\sqrt{Z_{t,i}} \Big|.
\end{equation}
It follows from (\ref{pla3.1}) to (\ref{pla3.4}) that
\begin{displaymath}
\widetilde{E}\Big( \big| \sqrt{\wtx_{t,i}}\,-\, \sqrt{\wtx_{t,i}'} \big| \;\Big|\; \wtx_{t-1}, \wtx_{t-1}' \Big)
\,\leq\, \sum_{j=1}^d \sqrt{B}_{ij} \;\; \big| \sqrt{\wtx_{t-1,j}} \,-\, \sqrt{\wtx_{t-1,j}'} \big| \;+\; b_i,
\end{displaymath}
where $b_i=2d+2\sup_{t\geq 1}E\big| \sqrt{Z_{t,i}} \,-\, E\sqrt{Z_{t,i}} \big|$.
Taking expectation on both sides, one can apply Lemma~\ref{spectral} with $C=\sqrt{B}$,
and $v_{t,i}=\widetilde{E}\Big( \big| \sqrt{\wtx_{t,i}}\,-\, \sqrt{\wtx_{t,i}'} \big| \Big)$ to conclude.
\end{proof}

{\lem
\label{LA.4}
Suppose that (\ref{m2.a}), (\ref{m2.b}), and (A2) are fulfilled.
Then there exists a coupling such that
\begin{displaymath}
\widetilde{E} \Big\| \sqrt{\wtx_{k+n-1}} \,-\, \sqrt{\wtx_{k+n-1}'} \Big\|_1
\,\leq\, \big\| (2\, \sqrt{B})^{n-1} \big\| \, \widetilde{E} \Big\| \sqrt{\wtx_k} \,-\, \sqrt{\wtx_k'} \Big\|_1.
\end{displaymath}
}
\bigskip

\begin{proof}[Proof of Lemma~\ref{LA.4}]
We apply a step-wise maximal coupling, that is, 
we choose $\wtz_t$ and $\wtz_t'$ such that they are equal,
and $\wty_t^{i,j}$ and $\wty_t^{i,j'}$ are coupled as described in Lemma~\ref{couplinglemma} below.
Then $\widetilde{P}\big( \wty_t^{i,j}\neq\wty_t^{i,j'} \;\big|\; \wtx_{t-1}, \wtx_{t-1}' \big)
\,=\, d_{TV}(P^{Y_t^{i,j}\mid X_{t-1}=\wtx_{t-1}}, P^{Y_t^{i,j}\mid X_{t-1}=\wtx_{t-1}'})$.
Note that, for $m>m'\geq 0$, $X\sim\mbox{Bin}(m,p)$ is stochastically larger than $X'\sim\mbox{Bin}(m',p)$.
Therefore, it follows that
$\wty_t^{i,j}\geq \wty_t^{i,j'}$ if $\wtl_t^{i,j}\geq \wtl_t^{i,j'}$
and, vice versa, $\wty_t^{i,j}\leq \wty_t^{i,j'}$ if $\wtl_t^{i,j}\leq \wtl_t^{i,j'}$; see c) of Lemma~\ref{couplinglemma} below.
Then, by Lemma~\ref{L_bin},
\begin{eqnarray*}
\widetilde{E}\Big( \big| \sqrt{\wtx_{t,i}} \,-\, \sqrt{\wtx_{t,i}'} \big| \;\Big|\; \wtx_{t-1},\wtx_{t-1}' \Big)
& \leq & \sum_{j=1}^d \widetilde{E}\Big( \big| \sqrt{\wty_t^{i,j}} \,-\, \sqrt{\wty_t^{i,j'}} \big| \;\Big|\; \wtx_{t-1},\wtx_{t-1}' \Big) \\
& \leq & \sum_{j=1}^d \Big| \widetilde{E}\Big( \sqrt{\wty_t^{i,j}} \,-\, \sqrt{\wty_t^{i,j'}} \;\Big|\; \wtx_{t-1},\wtx_{t-1}' \Big) \Big| \\
& \leq & \sum_{j=1}^d 2\, \sqrt{B_{ij}} \; \Big| \sqrt{\wtx_{t-1,j}} \,-\, \sqrt{\wtx_{t-1,j}'} \Big|.
\end{eqnarray*}
After taking expectation we can apply Lemma~\ref{spectral} with $b=0$, $C=2\sqrt{B}$ and
$v_{t,i}=\widetilde{E}\Big( \big| \sqrt{\wtx_{t,i}} \,-\, \sqrt{\wtx_{t,i}'}  \big|\Big)$
and obtain that
\begin{displaymath}
\widetilde{E} \Big\| \sqrt{\wtx_{k+n-1}} \,-\, \sqrt{\wtx_{k+n-1}'} \Big\|_1
\,\leq\, \big\| (2\sqrt{B})^{n-1} \Big\|_1 \, \widetilde{E} \Big\| \sqrt{\wtx_k} \,-\, \sqrt{\wtx_k'} \Big\|_1.
\end{displaymath}
\end{proof}
\bigskip

Lemmas~\ref{LA.3} and~\ref{LA.4} allow us to prove our second major result.

\begin{proof}[Proof of Theorem~\ref{T2}]
Let $k\in\N_0$, and let $(\wtx_t)_{t\in\N_0}$ and $(\wtx_t')_{t\in\N_0}$ be two versions of the count process,
where $(\wtx_0,\ldots,\wtx_k)$ and $(\wtx_0',\ldots,\wtx_k')$ are independent,
and where $\wtx_{k+1},\wtx_{k+2},\ldots$ are coupled with their respective counterparts $\wtx_{k+1}',\wtx_{k+2}',\ldots$
as described in the proof of Lemma~\ref{LA.4}.
Recall that the random variables $Y_{t,s}^{i,j}$ are independent which implies that
the terms $\sum_{s=1}^{X_{t-1,j}}Y_{t,s}^{i,j}$ ($i=1,\ldots,d$) are conditionally independent given~$X_{t-1}$,
and that $\sum_{s=1}^{X_{t-1,j}}Y_{t,s}^{i,j}\mid X_{t-1}\sim \mbox{Bin}(X_{t-1,j},B_{ij})$ .
Then we obtain from (\ref{1.2a}), Lemma~\ref{LA.3},  Lemma~\ref{LA.4}, Lemma~\ref{L5.3}, and (\ref{l1-sr})
\begin{eqnarray*}
\beta^X(k,n)
& \leq & \widetilde{P}\left( \wtx_{k+n} \neq \wtx_{k+n}' \right) \\
& \leq & \sum_{i=1}^d \widetilde{P}\left( \wtx_{k+n,i} \neq \wtx_{k+n,i}' \right) \\
& \leq & O\left( \widetilde{E}\big\| \sqrt{\wtx_{k+n-1}} - \sqrt{\wtx_{k+n-1}'} \big\|_1 \right) \\
& = & O\left( \big\| (2\sqrt{B})^{n-1} \big\|_1 \right) \,=\, O\big( \kappa^n \big).
\end{eqnarray*}
\end{proof}

\subsection{Proof of Theorem \ref{T2++}}
\label{SS4.4}

We consider two independent copies
$\widetilde{J}:=\left\{\left(\widetilde{Y}_t,\widetilde{Z}_t\right):t\in \N_0\right\}$ and
$\widetilde{J}':=\left\{\left(\widetilde{Y}'_t,\widetilde{Z}_t'\right):t\in \N_0\right\}$ of
$\left\{\left(Y_t,Z_t\right): t\in\N_0\right\}$,
with $Y_t=\left\{Y_{t,s}^{ij}: 1\leq i,j\leq d, s\in\N_0\right\}$ for any $t\in\N_0$.
We then define a process $\left(\widetilde{X}_t\right)_{t\in\N_0}$ defined from the recursions (\ref{m2.c})
and inputs $\widetilde{J}$ and another process $\left(\widetilde{X}'_t\right)_{t\in\N_0}$ also satisfying (\ref{m2.c}) but 
for which $\left(\widetilde{Y}_t,\widetilde{Z}_t\right)$ is replaced by 
$\left(\widetilde{Y}_t',\widetilde{Z}_t'\right)$ as soon as $t\geq k+1$.
For $t\in\N_0$ and $1\leq i\leq d$, set $v_{t,i}=\E\left\vert \widetilde{X}_{t,i}\right\vert$ and
$b_i=\sup_{s\geq 0}\E\vert Z_{s,i}\vert$. Since
\begin{displaymath}
v_{t,i} \,\leq\, \sum_{j=1}^d b_{ij}v_{t-1,j} \,+\, b_i,
\end{displaymath}
an application of Lemma~\ref{spectral} entails that $\sup_{t\geq 0}\Vert v_t\Vert_1<\infty$.
Moreover if $t\geq k+1$, setting $w_{t,i}=\E\left\vert \widetilde{X}_{t,i}-\widetilde{X}'_{t,i}\right\vert$, we get 
\begin{displaymath}
w_{t,i} \,\leq\, \sum_{j=1}^d B_{ij}w_{t-1,j},
\end{displaymath}
which leads to $\Vert w_t\Vert_1\leq \big\Vert B^{t-k}\big\Vert_1\Vert w_k\Vert_1$.
Moreover, the sequence $\left(w_k\right)_{k\in\N_0}$ is bounded. We conclude that 
\begin{eqnarray*}
\hspace*{4cm} \beta^X(k,n) & \leq & \P\left(\widetilde{X}_{k+n}\neq \widetilde{X}'_{k+n}\right) \\
& \leq & \sum_{i=1}^d\P\left(\widetilde{X}_{k+n,i}\neq \widetilde{X}'_{k+n,i}\right)\\
& \leq & \Vert w_{k+n}\Vert_1\\
& = & O\left( \kappa^n \right). \hspace*{6.2cm}\Box
\end{eqnarray*}

\subsection{Proof of Proposition~\ref{P3.1}}
\label{SS4.5}

\begin{proof}[Proof of Proposition~\ref{P3.1}]
It remains to prove (\ref{3.13a}), (\ref{3.13b}), (\ref{3.14a}), (\ref{3.14b}), and (\ref{3.15c}).
Recall that (\ref{3.12a}) provides the representation
\begin{displaymath}
X_t \,=\, b^t\, X_0 \,+\, \sum_{s=0}^{t-1} b^s\, \gamma_{t-s}
\,+\, \sum_{s=0}^{t-1} b^s\, \big( \varepsilon_{t-s} \,+\, Z_{t-s} \,-\, \gamma_{t-s} \big).
\end{displaymath}
First we derive upper estimates for the third term on the right-hand side of this equation which show
in particular that $X_t$ is dominated by its non-stochastic part.
We have that
\begin{eqnarray*}
E\big[ (\varepsilon_{t-s}+Z_{t-s}-\gamma_{t-s})^2 \big]
& = & E\big[ E\big( (\varepsilon_{t-s}+Z_{t-s}-\gamma_{t-s})^2 \mid \F_{t-s-1} \big) \big] \\
& = & E\big[ \var( \varepsilon_{t-s} \mid \F_{t-s-1} ) \,+\, \var( Z_{t-s} \mid \F_{t-s-1} ) \big] \\
& = & E\big[ \kappa\, X_{t-s-1} \,+\, \gamma_{t-s} \big] \\
& = & \big( \kappa/(1-b) \,+\, 1 \big) \, \gamma_{t-s} \, \big(1+o(1)\big),
\end{eqnarray*}
which leads to
\begin{displaymath}
\Big\| \sum_{s=0}^{t-1} b^s\, \big( \varepsilon_{t-s} \,+\, Z_{t-s} \,-\, \gamma_{t-s} \big) \Big\|_2
\,=\, O\big( \sqrt{\gamma_t} \big).
\end{displaymath}
Therefore, we obtain from (\ref{3.12}) that
\begin{eqnarray*}
\sum_{t=1}^n E\big[ X_{t-1}^2 \big]
& = & X_0^2 \,+\, \sum_{t=1}^{n-1} \big( EX_{t-1} \big)^2 \,+\, O\Big( \sum_{t=1}^{n-1} \gamma_t \Big) \\
& = & \frac{r_n}{(1-b)^2} \, \big( 1 \,+\, o(1) \big),
\end{eqnarray*}
i.e., (\ref{3.13a}) is proved.

Regarding higher moments, note first that, for $X\sim\mbox{Bin}(n,p)$,
\begin{displaymath}
E\big[ (X \,-\, EX)^4 \big]
\,=\, 3\big( np(1-p) \big)^2 \,+\, np(1-p)\big(1-6p(1-p)\big)
\end{displaymath}
(see e.g. \citet[Ch.~3.2, p.~107]{JKK92}) and, for $X\sim\mbox{Poi}(\lambda)$,
\begin{displaymath}
E\big[ (X \,-\, EX)^4 \big]
\,=\, 3\lambda^2 \,+\, \lambda
\end{displaymath}
(see e.g. \citet[Ch.~4.3, p.~157]{JKK92}).
This implies that
\begin{eqnarray*}
E\big[ (\varepsilon_{t-s}+Z_{t-s}-\gamma_{t-s})^4 \big]
& = & E\big[ E\big( (\varepsilon_{t-s}+Z_{t-s}-\gamma_{t-s})^4 \mid \F_{t-s-1} \big) \big] \\
& = & O\big( \gamma_{t-s}^2 \big)
\end{eqnarray*}
and, therefore,
\begin{displaymath}
\Big\| \sum_{s=0}^{t-1} b^s\, \big( \varepsilon_{t-s}+Z_{t-s}-\gamma_{t-s} \big) \Big\|_4
\,=\, O\big( \sqrt{\gamma_t} \big).
\end{displaymath}
This implies, again by (\ref{3.12}), that
\begin{eqnarray*}
\sum_{t=1}^n E\big[ X_{t-1}^3 \big]
& = & X_0^3 \,+\, \sum_{t=1}^{n-1} \big( EX_{t-1} \big)^3 \,+\, O\Big( \sum_{t=1}^{n-1} \gamma_t^{9/4} \Big) \\
& = & \frac{s_n}{(1-b)^3} \, \big( 1 \,+\, o(1) \big),
\end{eqnarray*}
i.e., (\ref{3.13b}) is proved.

In order to prove (\ref{3.14a}) and (\ref{3.14b}) to hold, we first truncate the $X_t$. Let
\begin{displaymath}
\bar{X}_t \,:=\, X_t \, \1\big( X_t \,\leq\, EX_t \,+\, \sqrt{\gamma_t}\, n^\delta \big),
\end{displaymath}
where $0<\delta<1/4$.
We obtain that
\begin{eqnarray*}
\lefteqn{ E\big[ X_t^2 \,-\, \bar{X}_t^2 \big] } \\
& = & E\big[ X_t^2 \, \1\big( \sum_{s=0}^{t-1} b^s( \varepsilon_{t-s} \,+\, Z_{t-s} \,-\, \gamma_{t-s})
\,>\, \sqrt{\gamma_t} \, n^\delta \big) \big] \\
& \leq & 2 \, \big( EX_t \big)^2 \, P\big( \sum_{s=0}^{t-1} b^s( \varepsilon_{t-s} \,+\, Z_{t-s} \,-\, \gamma_{t-s})
\,>\, \sqrt{\gamma_t} \, n^\delta \big) \\
& & {} \,+\, 2\, E\big[ \big( \sum_{s=0}^{t-1} b^s( \varepsilon_{t-s} \,+\, Z_{t-s} \,-\, \gamma_{t-s}) \big)^2
\, \1\big( \sum_{s=0}^{t-1} b^s( \varepsilon_{t-s} \,+\, Z_{t-s} \,-\, \gamma_{t-s})
\,>\, \sqrt{\gamma_t} \, n^\delta \big) \big] \\
& = & O\big( \gamma_t^2 \, E\big[ \big( \sum_{s=0}^{t-1} b^s( \varepsilon_{t-s} \,+\, Z_{t-s} \,-\, \gamma_{t-s}) \big)^4 \big]
/ (\gamma_t^2\,n^{4\delta}) \big) \\
& & {} \,+\, 2\, E\big[ \big( \sum_{s=0}^{t-1} b^s( \varepsilon_{t-s} \,+\, Z_{t-s} \,-\, \gamma_{t-s}) \big)^4 \big]
/ (\gamma_t\,n^{2\delta}) \big) \\
& = & O\big( \gamma_t^2\, n^{-4\delta} \,+\, \gamma_t\, n^{-2\delta} \big) \,=\, O\big( \gamma_t^2\, n^{-2\delta} \big),
\end{eqnarray*}
and therefore
\begin{subequations}
\begin{equation}
\label{pp31.1a}
\sum_{t=1}^n EX_{t-1}^2 \,-\, \sum_{t=1}^n E\bar{X}_{t-1}^2
\,=\, O\big( n^{-2\delta}\, \sum_{t=1}^n \gamma_{t-1}^2 \big) \,=\, O\big( r_n \, n^{-2\delta} \big).
\end{equation}
Since $X_t-\bar{X}_t$ is non-negative, we also obtain that
\begin{equation}
\label{pp31.1b}
\sum_{t=1}^n X_{t-1}^2 \,-\, \sum_{t=1}^n \bar{X}_{t-1}^2
\,=\, O_P\big( r_n \, n^{-2\delta} \big).
\end{equation}
Recall that the $\beta$-mixing coefficients serve as an upper bound for the $\alpha$-mixing coefficients.
Using a covariance inequality for $\alpha$-mixing random variables (see e.g.~\citet[Thm.~3, Sect.~1.2.2]{Dou94})
we obtain, for arbitrary $\tau>0$,
\begin{eqnarray}
\label{pp31.1c}
E\big[ \big( \sum_{t=1}^n \bar{X}_{t-1}^2 \,-\, \sum_{t=1}^n E\bar{X}_{t-1}^2 \big)^2 \big]
& = & \sum_{s,t=1}^n \cov\big( \bar{X}_{s-1}^2, \bar{X}_{t-1}^2 \big) \nonumber \\
& \leq & 8\, \sum_{s,t=1}^n \alpha(|s-t|)^{(2+\tau)/\tau} \, \big\| \bar{X}_{s-1}^2 \big\|_{2+\tau}
\, \big\| \bar{X}_{t-1}^2 \big\|_{2+\tau} \nonumber \\
& = & O\big( \sum_{s=1}^n (\gamma_{s-1}^2 + \gamma_{s-1} n^{2\delta})
\, (\gamma_{n-1}^2 + \gamma_{n-1} n^{2\delta}) \big) \nonumber \\
& = & O\big( r_n^2 \, n^{4\delta-1} \big).
\end{eqnarray}
(The last equality follows from the fact that $n\gamma_n^2=O(r_n)$.)
\end{subequations}
{}From (\ref{pp31.1a}) to (\ref{pp31.1c}) we obtain that
\begin{displaymath}
\sum_{t=1}^n X_{t-1}^2 \,-\, \sum_{t=1}^n EX_{t-1}^2 \,=\, o_P\big( r_n \big),
\end{displaymath}
which proves in conjunction with (\ref{3.13a}) relation (\ref{3.14a}).

(\ref{3.14b}) can be proved by analogous arguments.

Finally, we have to show that (\ref{3.15c}) is satisfied.
Note that, for $X\sim\mbox{Bin}(n,p)$, $E[(X-EX)^8]=O((np)^4)$ and, for $X\sim\mbox{Poi}(\lambda)$,
$E[(X-EX)^8]=O(\lambda^4)$; see \citet[Ch.~3.2 and Ch.~4.3]{JKK92}.
This implies that
$E[\varepsilon_t^8]=E[E(\varepsilon_t^8|\F_{t-1})]=O\big(E[E(X_{t-1}^4+\gamma_{t-1}^4|\F_{t-1})]\big)=O(\gamma_{t-1}^4)$.
Therefore we obtain
\begin{eqnarray*}
\sum_{t=1}^n E\big[ Z_{n,t}^{8/3} \big]
& = & \sum_{t=1}^n E\big[ X_{t-1}^{8/3}\, \varepsilon_t^{8/3}\, s_n^{-4/3} \big] \\
& \leq & \sum_{t=1}^n \Big( \underbrace{E\big[ X_{t-1}^4 \big]}_{=\,O(\gamma_{t-1}^4)} \Big)^{2/3}
\, \Big( \underbrace{E\big[ \varepsilon_t^8 \big]}_{=\,O(\gamma_{t-1}^4)} \Big)^{1/3} \, s_n^{-4/3} \\
& = & \sum_{t=1}^n O\big( \gamma_{t-1}^4 \, s_n^{-4/3} \big) \,=\, O\big( n^{-1/3} \big).
\end{eqnarray*}
\end{proof}

%%%%%%%%%%%%%%%%%%%%%%%%%%%%%%%%%%%%%%%%%%%%%%%%%%%%%%%%%%%%%%%%%%%%%%%%%%%%%%%
\section{A few auxiliary results}
\label{S5}
%%%%%%%%%%%%%%%%%%%%%%%%%%%%%%%%%%%%%%%%%%%%%%%%%%%%%%%%%%%%%%%%%%%%%%%%%%%%%%%

In this section we collect a few well-known facts which are used in the proofs of our main results.

{\lem
\label{couplinglemma}
Let $P$ and $P'$ be two probability distributions on $(\N_0,2^{\N_0})$ such that
$P(\{0,1,\ldots,k\})\leq P'(\{0,1,\ldots,k\})$ holds for all $k\in\N_0$. 
(If $X\sim P$ and $X'\sim P'$, then~$X$ is stochastically larger than~$X'$.)

Then there exist random variables $\widetilde{X}$ and $\widetilde{X}'$ on a common probability space
$\big(\widetilde{\Omega},\widetilde{\mathcal F},\widetilde{P}\big)$ such that
\begin{itemize}
\item[a)\quad] $\widetilde{P}^{\widetilde{X}}\,=\,P \qquad \mbox{and} \qquad \widetilde{P}^{\widetilde{X}'}\,=\,P'$,
\item[b)\quad] $\widetilde{P}\big(\widetilde{X}\neq\widetilde{X}'\big)\,=\,d_{TV}\big(P,P'\big)$,
\item[c)\quad] $\widetilde{P}\big(\widetilde{X}\geq\widetilde{X}'\big)\,=\,1$.
\end{itemize}
}

\begin{proof}
Let~$p$ and~$p'$ be the probability mass functions of~$P$ and~$P'$, respectively.
Recall that 
\begin{displaymath}
d_{TV}\big(P,Q\big) \,=\, (1/2)\, \sum_{k=0}^\infty \big| p(k) \,-\, q(k) \big|
\,=\, 1 \,-\, \sum_{k=0}^\infty p(k)\wedge q(k).
\end{displaymath}
(Note that our definition of the total variation norm differs from that in \citet{Lin92} by the factor 2.)

\citet[Theorem~5.2 in Chapter~I]{Lin92} describes a widely used approach where random variables~$\widetilde{X}$ and~$\widetilde{X}'$
are defined such that 
\begin{equation}
\label{pl51.1}
\widetilde{P}\big( \widetilde{X}=\widetilde{X}'=k \big) \,=\, p(k)\wedge p'(k) \qquad \forall k\in\N_0,
\end{equation}
while $\widetilde{X}$ and $\widetilde{X}'$ are independent on the event $\{\widetilde{X}\neq\widetilde{X}'\}$.
Then a) and b) are satisfied, however, requirement c) is not fulfilled in general.
We propose here a slight modification of this method. It is obvious that a) and b) necessarily require
a coupling satisfying (\ref{pl51.1}). In order to satisfy c) also, we use a quantile transform to generate~$\widetilde{X}$ and~$\widetilde{X}'$
on $\{\widetilde{X}\neq\widetilde{X}'\}$. In what follows we describe our approach in a more formal way.

Suppose that $\big(\widetilde{\Omega},\widetilde{\mathcal F},\widetilde{P}\big)$ admits a random
variable~$\widetilde{Z}$ which has a uniform distribution on $(0,1)$.
We define distribution functions $F$, $G$ and $G'$ on~$\N_0$ as
\begin{eqnarray*}
F(k) & = & \sum_{l=0}^k p(l)\wedge p'(l), \\
G(k) & = & \sum_{l=0}^k p(l) \,-\, p(l)\wedge p'(l), \\
G'(k) & = & \sum_{l=0}^k p'(l) \,-\, p(l)\wedge p'(l).
\end{eqnarray*}
Then
\begin{displaymath}
\lim_{k\to\infty} F(k) \,=\, 1 \,-\, d_{TV}(P, P') \,=:\, \gamma.
\end{displaymath}
If $\widetilde{Z}\leq \gamma$, then we define
\begin{displaymath}
\widetilde{X} \,=\, \widetilde{X}' \,:=\, F^{-1}(\widetilde{Z}),
\end{displaymath}
where $H^{-1}(t)=\inf\{x\colon\, H(x)\geq t\}$ denotes the generalized inverse of a generic function~$H$.
Otherwise, if $\widetilde{Z}>\gamma$, then we define
\begin{displaymath}
\widetilde{X} \,:=\, G^{-1}(\widetilde{Z}-\gamma) \qquad \mbox{ and } \qquad \widetilde{X}' \,:=\, G'^{-1}(\widetilde{Z}-\gamma).
\end{displaymath}
It follows that
\begin{displaymath}
\widetilde{P}^{\widetilde{X}}\,=\,P \qquad \mbox{and} \qquad \widetilde{P}^{\widetilde{X}'}\,=\,P'.
\end{displaymath}
Furthermore, we have that
\begin{displaymath}
\widetilde{P}\big( \widetilde{X}=\widetilde{X}' \big) \,\geq\, \widetilde{P}\big( \widetilde{Z}\leq \gamma \big) \,=\, \gamma.
\end{displaymath}
On the other hand, the well-known coupling inequality (see Section~I.2 in \citet{Lin92}) entails that
$\widetilde{P}\big(\widetilde{X}=\widetilde{X}'\big)\leq\gamma$, which proves b).
Finally, it follows from $G(k)\leq G'(k)$ $\forall k\in\N_0$ that $G^{-1}(z)\geq G'^{-1}(z)$ $\forall z$.
Therefore, c) is also satisfied.
\end{proof}

{\lem
\label{L.inequality}
Let $0\leq \lambda<\lambda'$, and let $X\sim\mbox{\rm Poi}(\lambda)$, $X'\sim\mbox{\rm Poi}(\lambda')$.
Then
\begin{displaymath}
E\sqrt{ X' } \,-\, E\sqrt{ X }
\,\leq\, 2\big(\sqrt{\lambda'} \,-\, \sqrt{\lambda}\big).
\end{displaymath}
}

\begin{proof}[Proof of Lemma~\ref{L.inequality}]
Let $(X_\nu)_{\nu>0}$ be a collection of Poisson variates with respective intensities~$\nu$.
Then the function $\nu\mapsto E\sqrt{X_\nu}$ is differentiable and it holds
\begin{subequations}
\begin{eqnarray}
\label{pl62.1a}
\frac{d}{d\nu} E\sqrt{X_\nu} 
& = & \sum_{k=1}^\infty \sqrt{k} \; \frac{d}{d\nu}\Big\{ e^{-\nu} \frac{\nu^k}{k!} \Big\} \nonumber \\
& = & \sum_{k=1}^\infty \sqrt{k} \; \big(\frac{k}{\nu}\,-\,1\big) \; e^{-\nu} \frac{\nu^k}{k!} \nonumber \\
& = & \sum_{k=1}^\infty \sqrt{k} \, e^{-\nu} \frac{\nu^{k-1}}{(k-1)!} \quad - \quad \sum_{k=1}^\infty \sqrt{k} \, e^{-\nu} \frac{\nu^k}{k!} \nonumber \\
& = & E\sqrt{X_\nu+1} \,-\, E\sqrt{X_\nu}.
\end{eqnarray}
Since the function $x\mapsto\sqrt{x+1}$ is concave we obtain by Jensen's inequality that
\begin{equation}
\label{pl62.1b}
E\sqrt{X_\nu+1} \,\leq\, \sqrt{\nu+1}.
\end{equation}
Furthermore, since
$E\sqrt{X_\nu}=\sum_{k=1}^\infty \sqrt{k}e^{-\nu}\nu^k/k! = \nu\sum_{l=0}^\infty (1/\sqrt{l+1})\,e^{-\nu}\nu^l/l!=\nu E[1/\sqrt{X_\nu+1}]$
and since the function $x\mapsto 1/\sqrt{x+1}$ is convex we obtain again by Jensen's inequality that
\begin{equation}
\label{pl62.1c}
E\sqrt{X_\nu} \,\geq\, \frac{\nu}{\sqrt{EX_\nu \,+\, 1}} \,=\, \frac{\nu}{\sqrt{\nu+1}}.
\end{equation}
\end{subequations}
It follows from (\ref{pl62.1a}) to (\ref{pl62.1c}) that
\begin{displaymath}
\frac{d}{d\nu} E\sqrt{X_\nu} \,\leq\, \sqrt{\nu+1} \,-\, \frac{\nu}{\sqrt{\nu+1}} \,=\, \frac{1}{\sqrt{\nu+1}}.
\end{displaymath}
This implies
\begin{displaymath}
E\sqrt{X'} \,-\, E\sqrt{X} \,=\, \int_\lambda^{\lambda'} \frac{d}{du} E\sqrt{X_u} \, du
\,\leq\, \int_\lambda^{\lambda'} \frac{1}{\sqrt{u}}\, du \,=\, 2\big( \sqrt{\lambda'} \,-\, \sqrt{\lambda} \big).
\end{displaymath}
\end{proof}

{\lem
\label{L_bin}
Let $X_n\sim \mbox{Bin}(n,p)$, where $n\in\N$, $p\in[0,1]$. Then
\begin{displaymath}
E\sqrt{X_{n+1}} \,-\, E\sqrt{X_n} \,\leq\, 2\,\sqrt{p}\, \big(\sqrt{n+1} \,-\, \sqrt{n}\big) \qquad \forall n\in\N.
\end{displaymath}                                                     
}

\begin{proof}
We have that
\begin{eqnarray*}
E\sqrt{X_{n+1}} \,-\, E\sqrt{X_n}
& = & p\, E\big[ \sqrt{X_{n}+1} \,-\, \sqrt{X_n} \big] \\
& = & p\, E\big[ \frac{1}{ \sqrt{X_{n}+1} \,+\, \sqrt{X_n} } \big]
\,\leq\, p\, E\big[ \frac{1}{ \sqrt{X_{n}+1} } \big].
\end{eqnarray*}
Furthermore,
\begin{eqnarray*}
E\big[ \frac{1}{ \sqrt{X_{n}+1} } \big]
& = & \sum_{k=0}^n \frac{1}{\sqrt{k+1}} \, {n \choose k} \, p^k \, (1-p)^{n-k} \\
& = & \sum_{k=0}^n \sqrt{k+1} \, \frac{n!}{(n-k)!\, (k+1)!} \, p^k \, (1-p)^{(n+1)-(k+1)} \\
& = & \frac{1}{p(n+1)} \, \sum_{l=1}^{n+1} \sqrt{l} \, {n+1 \choose l} \, p^l \, (1-p)^{n+1-l} \\
& = & \frac{1}{p(n+1)} \, E\sqrt{ X_{n+1} } \,\leq\, \frac{1}{\sqrt{p(n+1)}},
\end{eqnarray*}
which implies that
\begin{displaymath}
E\sqrt{X_{n+1}} \,-\, E\sqrt{X_n} \,\leq\, \frac{\sqrt{p}}{\sqrt{n+1}}
\,\leq\, \sqrt{p}\, \frac{2}{\sqrt{n+1}\,+\,\sqrt{n}} \,=\, 2\, \sqrt{p} \, \big( \sqrt{n+1} \,-\, \sqrt{n} \big).
\end{displaymath}
\end{proof}

{\lem
\label{L5.3}
If $p_0<1$, then
\begin{equation}
\label{1.4}
\sup_{p\leq p_0} d_{TV}\big( \mbox{Bin}(n,p), \mbox{Bin}(m,p) \big) \,=\, O\big( |\sqrt{n} \,-\, \sqrt{m}| \big).
\end{equation}
} 

\begin{proof}
Let, w.l.o.g., $n>m$. 
We denote by $f(\cdot\, ;N,p)$ the probability mass function of a binomial distribution with parameters~$N$ and~$p$,
i.e. $f(k;N,p)\,=\,{N\choose k}p^k(1-p)^{N-k}$ for $k=0,1,\ldots,N$ and $f(k;N,p)\,=\,0$ otherwise.
Then, for fixed~$p$, the mapping
\begin{displaymath}
k \,\mapsto\, \frac{f(k;m,p)}{f(k;n,p)} \,=\, \frac{m(m-1)\;\cdots\;(m-k+1)}{n(n-1)\;\cdots\;(n-k+1)}\; (1-p)^{m-n}
\end{displaymath}
is non-increasing on $\{0,1,\ldots,n\}$, where $f(0;m,p)/f(0;n,p)=(1-p)^{m-n}>1$ and \mbox{$f(n;m,p)/f(n;n,p)=0$}.
Let $k_0(p):=\min\{k\colon \; f(k;m,p)/f(k;n,p)<1\}$.
Then
\begin{eqnarray*}
d_{TV}\big( \mbox{Bin}(n,p), \mbox{Bin}(m,p) \big)
& = & \frac{1}{2} \; \sum_{k=0}^n \big| f(k;n,p) \,-\, f(k;m,p) \big| \\
& = & \sum_{k\colon\, f(k;n,p)>f(k;m,p)} f(k;n,p) \,-\, f(k;m,p) \\
& = & \sum_{k=k_0(p)}^n f(k;n,p) \,-\, f(k;m,p).
\end{eqnarray*}
To handle the supremum we next show that
\begin{equation}
\label{pl11.1}
\sup_{p\leq p_0} d_{TV}\big( \mbox{Bin}(n,p), \mbox{Bin}(m,p) \big)
\,=\, d_{TV}\big( \mbox{Bin}(n,p_0), \mbox{Bin}(m,p_0) \big).
\end{equation}
Note that, for fixed $m$, $n$ and $k$, the mapping $p\mapsto f(k;m,p)/f(k;n,p)$ is non-decreasing,
which implies that $p\mapsto k_0(p)$ is a non-decreasing and piecewise constant function. 
Denote the discontinuity points of this function by $p_1,\ldots,p_K$.
For $p\not\in\{p_1,\ldots,p_K\}$, we have that
\begin{eqnarray*}
\lefteqn{ \frac{d}{dp} \sum_{k=k_0(p)}^n f(k;n,p) } \\
& = & \sum_{k=k_0(p)}^{n-1} \Big\{ \frac{n!}{(n-k)!(k-1)!} \; p^{k-1} \; (1-p)^{n-k} 
\,-\, \frac{n!}{(n-(k+1))!k!} \; p^k \; (1-p)^{n-(k+1)} \Big\}
\;+\; {n \choose n} \, n \; p^{n-1} \\
& = & \frac{k_0(p)}{p} \; {n \choose k_0(p)} \; p^{k_0(p)} \; (1-p)^{n-k_0(p)}
\end{eqnarray*}
and, analogously, 
\begin{eqnarray*}
\lefteqn{ \frac{d}{dp} \sum_{k=k_0(p)}^n f(k;m,p) } \\
& = & \frac{k_0(p)}{p} \; {m \choose k_0(p)} \; p^{k_0(p)} \; (1-p)^{m-k_0(p)}.
\qquad \qquad \qquad \qquad \qquad \qquad \qquad \qquad \qquad \qquad
\end{eqnarray*}
which implies that
\begin{displaymath}
\frac{d}{dp} d_{TV}\big( \mbox{Bin}(n,p), \mbox{Bin}(m,p) \big)
\,=\, \frac{k_0(p)}{p} \Big( f(k_0(p);n,p) \,-\, f(k_0(p);m,p) \big)
\,>\, 0 \quad \forall p\not\in\{p_1,\ldots,p_K\}.
\end{displaymath}
(Note that $k_0(p)>0$.) Hence, (\ref{pl11.1}) holds true.

For $N\in\N$, let $S_N\sim\mbox{Bin}(N,p_0)$.
Then we obtain from the Berry-Esseen inequality that
\begin{displaymath}
\sup_x \Big| P\Big( \frac{S_N-Np_0}{\sqrt{Np_0(1-p_0)}} \leq x \Big) \,-\, \Phi(x) \Big| \,=\, O\big( \frac{1}{\sqrt{N}} \big).
\end{displaymath}
Using this approximation we obtain that
\begin{eqnarray}
\label{pl11.2}
\lefteqn{ d_{TV}\big( \mbox{Bin}(n,p_0), \mbox{Bin}(m,p_0) \big) } \nonumber \\
& = & \big| P\big( S_n \leq k_0(p_0)-1 \big) \;-\; P\big( S_m \leq k_0(p_0)-1 \big) \big| \nonumber \\
& = & \Big| \Phi\Big( \frac{(k_0(p_0)-1)-np_0}{\sqrt{np_0(1-p_0)}} \Big) 
\,-\, \Phi\Big( \frac{(k_0(p_0)-1)-mp_0}{\sqrt{mp_0(1-p_0)}} \Big) \Big| 
\,+\, O\Big( \frac{1}{\sqrt{n}} \,+\, \frac{1}{\sqrt{m}} \Big) \nonumber \\
& \leq & \frac{1}{\sqrt{2\pi}} \; \Big| \frac{(k_0(p_0)-1)-np_0}{\sqrt{np_0(1-p_0)}}
\,-\, \frac{(k_0(p_0)-1)-mp_0}{\sqrt{mp_0(1-p_0)}} \Big|
\,+\, O\Big( \frac{1}{\sqrt{n}} \,+\, \frac{1}{\sqrt{m}} \Big) \nonumber \\
& = & O\big( | \sqrt{n} \,-\, \sqrt{m} | \big).
\end{eqnarray}
The assertion of the lemma follows from (\ref{pl11.1}) and (\ref{pl11.2}).
\end{proof}

\begin{lem}\label{spectral}
Suppose that $\{v_{t,i}: t\geq 0, 1\leq i\leq d\}$ is family of non-negative real numbers
such that there exist a matrix~$C$ of size $d\times d$, with non-negative entries~$C_{ij}$ and a vector~$b$ of~$\R^d$
with non-negative entries~$b_i$ such that
\begin{displaymath}
v_{t,i} \,\leq\, \sum_{j=1}^d C_{ij}v_{t-1,j} \,+\, b_i, \qquad 1\leq i\leq d, t\geq 1.
\end{displaymath}
Then 
\begin{displaymath}
v_t \,\preceq\, \sum_{s=0}^{t-1}C^sb \,+\, C^t v_0,
\end{displaymath}
where $\preceq$ denotes the coordinatewise ordering on $\R^d$ (i.e. $v\preceq v'$ means that $v_j\leq v'_j$ for $1\leq j\leq d$). 
If $\rho(C)<1$, then for any $\kappa\in(\rho(C),1)$ there exists $\alpha<\infty$ such that
\begin{displaymath}
\|v_t\|_1 \,\leq\, \alpha\, \Big( \sum_{s=0}^\infty \kappa^s\, \|b\|_1 \,+\, \kappa^t\, \|v_0\|_1 \Big).
\end{displaymath}
\end{lem}

\begin{proof}
Using non-negativity of the coefficients and iterating the bound for the $v_{t,i}'$s we obtain
\begin{displaymath}
v_t \,\preceq\, Cv_{t-1} \,+\, b \,\preceq\, C\,(Cv_{t-2} \,+\, b) \,+\, b
\,\preceq\, \ldots \,\preceq\, C^t\, v_0 \,+\, \big( C^0+\cdots +C^{t-1} \big)\, b.
\end{displaymath}
It follows from Gelfand's formula $\lim_{s\to\infty}\|C^s\|_1^{1/s}=\rho(C)$ that $\|C^s\|_1\leq\kappa^s$ for $\kappa>\rho(C)$
and~$s$ sufficiently large, which implies that $\|C^s\|_1\leq \alpha\kappa^s$ for all $s\in\N$ and some $\alpha<\infty$.
This implies
\begin{displaymath}
\| v_t \|_1
\,\leq\, \sum_{s=0}^\infty \big\| C^t \|_1\, \| b \|_1 \,+\, \big\| C^t \big\|_1 \, \| v_0 \|_1
\,\leq\, \alpha\, \Big( \sum_{s=0}^\infty \kappa^s\, \|b\|_1 \,+\, \kappa^t\, \|v_0\|_1 \Big).
\end{displaymath}
which completes the proof.
\end{proof}
\bigskip

\begin{ack}
This work was funded by Project ``EcoDep'' PSI-AAP2020 -- 0000000013.
The authors thank a Co-Editor and two anonymous referees for their valuable comments that led to a significant improvement of the paper.
\end{ack}
\bigskip

\begin{center}
Data Availability Statement
\end{center}
\noindent
The data that support the findings of this study are available on 
\url{https://www.kaggle.com/datasets/aishu200023/stackindex}.
\bigskip

\bibliographystyle{harvard}

\end{document}